\documentclass[a4paper]{amsart}
\usepackage{amssymb,amsmath,amsthm,amscd}
\usepackage{mathrsfs}
\usepackage[mathscr]{euscript}
\usepackage[all]{xy}
\usepackage[utf8]{inputenc}
\usepackage{lmodern}
\usepackage[T1]{fontenc}
\usepackage[textwidth=160mm,hcentering]{geometry}
\usepackage[pagebackref=true,breaklinks=true,letterpaper=true,colorlinks]{hyperref}
\theoremstyle{plain}
\newtheorem{theo}{Theorem}[section]

\newtheorem{prop}[theo]{Proposition}

\newtheorem{coro}[theo]{Corollary}

\theoremstyle{definition}
\newtheorem{defi}[theo]{Definition}
\newtheorem{nota}[theo]{Notation}
\newtheorem{cons}[theo]{Construction}

\theoremstyle{remark}
\newtheorem{rem}[theo]{Remark}
\newtheorem{exa}[theo]{Example}

\numberwithin{equation}{section}

\newcommand{\op}{^{\mathrm{op}}}
\newcommand{\cat}{\mathbf}
\newcommand{\oper}{\mathscr}
\newcommand{\on}{\mathrm}
\newcommand{\Emb}{\on{Emb}}
\newcommand{\Map}{\on{Map}}
\newcommand{\Mod}{\cat{Mod}}

\newcommand{\oMod}{\oper{M}od}
\newcommand{\oMan}{\oper{M}an}
\newcommand{\oSpace}{\oper{S}pace}
\newcommand{\oCom}{\oper{C}om}
\newcommand{\oAss}{\oper{A}ss}
\newcommand{\oMor}{\oper{M}or}
\newcommand{\oCob}{\oper{C}ob}
\newcommand{\oCospan}{\oper{C}ospan}
\newcommand{\oModCat}{\oMod\oper{C}at}
\newcommand{\Cat}{\cat{Cat}}
\newcommand{\Alg}{\cat{Alg}}
\newcommand{\Man}{\cat{Man}}
\renewcommand{\S}{\cat{S}}

\newcommand{\Spec}{\cat{Spec}}

\newcommand{\un}{\mathbb{I}}
\newcommand{\id}{\mathrm{id}}
\newcommand{\goto}[1]{\stackrel{#1}{\longrightarrow}}

\newcommand{\Hom}{\underline{\mathrm{Hom}}}

\renewcommand{\L}{\mathbb{L}}

\newcommand{\undtil}[1]{\underline{\widetilde{#1}}}
\newcommand{\und}[1]{\underline{#1}}

\title{Operads, modules and Topological field theories}
\author{Geoffroy Horel}

\begin{document}

\address{Mathematisches Institut\\
Einsteinstrasse 62\\
D-48149 Münster\\
Deutschland}
\email{geoffroy.horel@gmail.com}\thanks{The author was partially supported by an NSF grant and Michael Weiss's Alexander von Humboldt professor grant}

\begin{abstract}
In this paper, we describe a general theory of modules over an algebra over an operad. We also study functors between categories of modules. Specializing to the operad $\oper{E}_d$ of little $d$-dimensional disks, we show that each $(d-1)$ manifold gives rise to a theory of modules over $\oper{E}_d$-algebras and each bordism gives rise to a functor from the category defined by its incoming boundary to the category defined by its outgoing boundary. We describe how to assemble these categories into a map from a certain $\infty$-operad to the $\infty$-operad of $\infty$-categories.
\end{abstract}

\maketitle

\tableofcontents

\section*{Introduction}

A standard idea in mathematics is to study algebras through their representations, also known as modules. This idea can be applied to various notions of algebras (associative algebras, commutative algebras, Lie algebras, etc.). If we have to deal with more complicated types of algebras defined by an operad, we must first understand what the correct notion of module is. There is a definition of \emph{operadic modules} over an algebra over an operad, but this is too restrictive in our opinion. For instance, operadic modules over associative algebras are bimodules. However, left modules are at least equally interesting as bimodules. This suggests that, in general, there are several interesting theories of modules over an algebra. 

Our first contribution in this paper is to classify all objects that can sensibly be called modules over an algebra over a certain operad. As it turns out, for a given operad $\oper{O}$, notions of modules over $\oper{O}$-algebras are in one-to-one correspondence with associative algebras in the symmetric monoidal category of right $\oper{O}$-modules (see \ref{PMod}). For $P$ an associative algebra in right modules over $\oper{O}$, we say that a module parametrized by this particular object is a $P$-shaped module. For instance left modules, right modules and bimodules are three different shapes of modules for the operad $\oper{A}ss$.

An interesting feature of the categories of modules of a certain shape over a certain kind of algebra is that they usually carry operations that are present, independently of what the algebra is. Those operations are entirely determined by the type of algebra and the shape of the module. For example, if one takes a commutative algebra, then the category of left modules has a symmetric monoidal structure. If the algebra is only associative, then this symmetric monoidal structure does not exist. On the other hand, on the category of bimodules over an associative algebra, there exists a monoidal structure.

Monoidal or symmetric monoidal structure on categories are not specific to homotopy theory and can be found in most fields of mathematics. However, when working in a homotopy theoretic context, one may encounter monoidal structure parametrized by operads in spaces. For instance, it has been proved by Lurie in \cite{luriehigher} that if $A$ is an $\oper{E}_{d+1}$-algebra, there is an $\oper{E}_{d}$-monoidal structure on the category of left modules over $A$. If $d$ is at least $3$, there is no classical analogue of an $\oper{E}_d$-monoidal category. A similar kind of result, also due to Lurie, is that the category of operadic $\oper{E}_d$-modules has an $\oper{E}_d$-monoidal structure. This last result is the main step in Lurie's proof of Deligne's conjecture.

Our second main contribution in this paper is to construct operations on categories of modules over algebras of a certain type that generalize all those that we have just mentioned. 

Before explaining these operations, let us say a few words about our language. The paper \cite{luriehigher} uses $\infty$-categorical techniques. We have decided to use model categories instead. Most models of $\infty$-categories admit a strict enrichment in spaces, which allows one to speak of an $\infty$-category which is $\oper{O}$-monoidal for some operad in spaces $\oper{O}$. On the other hand, as far as we know, there is no accepted definition of an $\oper{O}$-monoidal model category. In this paper we suggest a definition of such an object by constructing a simplicial operad of model categories (see \ref{construction of oModCat}). 

The objects of this operad are model categories, the morphisms are given by left Quillen multi-functors and weak equivalences between them. We also extends Rezk's nerve to a functor from the operad of model category to the operad of complete Segal spaces (see \ref{comparison model category complete Segal spaces}). Hence, if we have an operad $\oper{O}$, we can make sense of what an $\oper{O}$-algebra in model category is. We just define it to be a map from $\oper{O}$ to our operad of model categories. Using our comparison map we see that such a data induces an $\oper{O}$-algebra structure on the corresponding complete Segal space, which means that our theory is homotopically sensible.

Coming back to our initial problem, our approach is to put a model structure on the categories of modules we have constructed (see \ref{model category on PMod_A}). This first step is quite standard. At this stage, for a given operad $\oper{O}$ and a given (cofibrant) $\oper{O}$-algebra $A$, we are able to construct a function $P\mapsto P\Mod_A$ which sends an associative algebra in right $\oper{O}$-modules to a model category of $P$-shaped $A$-modules.

Our next step is to extends this function to a map of operad from a simplicial operad $\oMor(\oper{O})$ to the operad of model categories (see \ref{map from Mor to ModCat}). The operad $\oMor(\oper{O})$ is the Morita operad of $\Mod_{\oper{O}}$. Its objects are associative algebras in $\Mod_{\oper{O}}$ and its morphisms are given by bimodules and weak equivalences between them.

As a particular example we study the case of $\oper{E}_d$-algebra. We construct categories of modules associated to $(d-1)$-manifolds (see \ref{def of Stau modules}) and construct functors between these categories of modules indexed by bordisms (see \ref{functor induced by a bordism}). In the end, the structure we produce is a map of operad from a certain operad $f\widehat{\oper{C}ob}_d$ (defined in \ref{fake cobordism category}) closely related to the symmetric monoidal category of cobordims to the operad of model categories. As a corollary, we recover the fact that the category of operadic $\oper{E}_d$-modules over an $\oper{E}_d$-algebra is an $\oper{E}_d$-monoidal category and that the category of left modules is an $\oper{E}_{d-1}$-monoidal category.

We also study the case of commutative algebras. In that case we show that factorization homology reduces to the tensor product between spaces and commutative algebras generalizing a result of \cite{mcclurethh}. We also construct operations on the various categories of modules over a commutative algebra indexed by cospans of spaces (see \ref{from cospan to modcat}).

\subsection*{Related work}

The idea of using right $\oper{O}$-modules to construct interesting invariants of $\oper{O}$-algebras was initiated in \cite{fressemodules}.

The idea of a $2$-category of model categories is mentioned without any definition in \cite{hoveymodel}. It is also implicit in several papers of Dugger.

In \cite{luriehigher}, the author shows that the category of operadic $\oper{E}_d$-module over an $\oper{E}_d$-algebra carries an action of the operad $\oper{E}_d$. Our work extends this action to an action of the operad of cobordisms from copies of the $(d-1)$-spheres to the $(d-1)$-sphere.

The recent paper \cite{toenoperations} shows that for nice operads $\oper{O}$ in spaces, the operad $\oper{O}$ acts on $\oper{O}(2)$ via cospans of spaces. In our language, we see that if $A$ is a commutative algebra in $\cat{C}$, then $\int_{\oper{O}(2)}A$ is an $\oper{O}$-algebra in $\oMor(\cat{C})$. In particular, $L\Mod_{\int_{\oper{O}(2)}A}$ is an $\oper{O}$-algebra in $\oModCat$.

The existence of a fully extended topological field theory constructed from an $\oper{E}_d$-algebra was sketched in \cite{lurieclassification}. A rigorous construction will appear in  \cite{calaquefactorization} and \cite{calaquecategory}. In this paper, we construct the restriction of this field theory in dimension $d$ and $d-1$.

\subsection*{Acknowledgments}
This paper was mainly developed during my time as a graduate student at MIT. I would like to thank my advisor Haynes Miller for his guidance. This work also benefited a lot from conversations with Clark Barwick, David Ayala, Ricardo Andrade and Pedro Boavida de Brito.

\subsection*{Conventions}

\begin{itemize}
\item A boldface letter or word like $\cat{X}$ or $\cat{Mod}$ always denotes a category. 

\item All categories are assumed to be simplicial. If they are ordinary categories we give them the discrete simplicial structure. We denote by $\on{Fun}(\cat{X},\cat{Y})$ the simplicial category of simplicial functors from $\cat{X}$ to $\cat{Y}$. 

\item $\Map_{\cat{X}}(X,Y)$ denotes the simplicial set of maps between $X$ and $Y$ in the category $\cat{X}$.

\item $\cat{X}(X,Y)$ denotes the set of maps from $X$ to $Y$ in the category $\cat{X}$. Equivalently, $\cat{X}(X,Y)$ is the set of $0$-simplices of $\Map_{\cat{X}}(X,Y)$.

\item A calligraphic letter like $\oper{M}$ always denotes a (colored) operad in the category of simplicial sets.

\item If $\cat{C}$ is a symmetric monoidal simplicial categoy, $\cat{C}[\oper{M}]$ denotes the category of $\oper{M}$-algebras in $\cat{C}$.

\item The symbol $\cong$ denotes an isomorphism. The symbol $\simeq$ denotes an isomorphism in the homotopy category (i.e. a zig-zag of weak equivalences). 

\item The letters $Q$ and $R$ generically denote the cofibrant and fibrant replacement functor in the ambient model category. There is a natural transformation $Q\to\on{id}$ and $\on{id}\to R$.

\item In this work, the word space usually means \emph{simplicial set}. We try to say \emph{topological spaces} when we want to talk about topological spaces.

\item We allow ourselves to treat topological spaces as simplicial sets without changing the notation. The reader is invited to apply the functor $\on{Sing}$ as needed.

\item We also allow ourselves to treat category as simplicial sets without changing the notation. More precisely, if we have a category or an operad enriched in categories, we use the same notation for the simplicially enriched category or operad obtained by applying the nerve functor to each Hom category.

\item The word \emph{spectrum} is to be interpreted as symmetric spectrum in simplicial sets.

\item We say \emph{large category} to talk about a category enriched over possibly large simplicial sets. We say \emph{category} to talk about a category enriched over small simplicial sets. We say \emph{small category} to talk about a category whose objects and morphisms both are small. The meaning of small and large can be made precise by way of Grothendieck universes.
\end{itemize}

\section{The Morita bioperad}

\subsection{Bioperads}

In this subsection, we develop a theory of bioperads. A bioperad is to an operad in $\Cat$ what a bicategory is to a $2$-category. Bioperads with one object are studied under the name operadic category in \cite{toenoperations} \footnote{We have decided to change the name because the term operadic category is used in another context in work of Batanin and Markl}.

We sue the notation $\Omega$ for the category of dendrices (see for instance \cite{moerdijkinner}). For $\tau\in\Omega$, we denote $E(\tau)$ the set of edges of $\tau$, $L(\tau)$ the set of leaves and $R(\tau)$ the root. The set $E(\tau)-(L(\tau)\sqcup R(\tau))$ is called the set of internal edges and is the set of edges connected to exactly two vertices. We denote by $V(\tau)$ the set of vertices. If $v$ is a vertex of a tree, we denote by $v_{in}$ the set of incoming edges and $v_{out}$ the unique outgoing edge. A vertex is called external if all of its incoming edges are leaves. Recall that the morphisms in $\Omega\op$ are compositions of isomorphisms, faces and degeneracies. The degeneracies ``blow-up'' an edge into two edges connected by a vertex. The faces either collapse an internal edge or remove an external vertex together with all its incoming edges.

\begin{defi}
If $S$ is a set, an \emph{$S$-decorated tree} is the data of an element of $\Omega$ together with a map $S\to E(\tau)$.
\end{defi}

\begin{defi}
An \emph{$S$-multigraph} in $\cat{Cat}$ is the data, for each family of elements of $S$, $\{x_i\}_{i\in I}$ and each element $y$ of $S$, of a category $\oper{G}(\{x_i\}_{i\in I};y)$.
\end{defi}

An $S$-multi-graph can assign a value to any $S$-decorated tree by the following formula
\[\oper{G}(\tau)=\prod_{v\in V(\tau)}\oper{G}(\{e\}_{e\in v_{in}};v_{out})\]
Note that $\tau$ being decorated, any edge is labeled by an element of $S$, $v_{in}$ denotes the set of incoming edges with their given label and similarly for $v_{out}$.

If $x:\tau\to\tau'$ is a map in $\Omega\op$ any decoration on $\tau$ can be transferred to a decoration on $\tau'$. Indeed, if $x$ is a face or an isomorphism, this is tautological and if $x$ is a degeneracy (i.e. turns an edge into two edges connected by a vertex), we just duplicate the label on the edge.

\begin{defi}
Let $\tau$ be a decorated tree, a \emph{composition data on $\tau$} is a sequence $c=(c_0,\ldots,c_k)$ of composable faces, degeneracies and isomorphisms in $\Omega\op$ such that the source of $c_0$ is $\tau$.
\end{defi}

If $c$ is a composition data on $\tau$, we denote by $c_*(\tau)$ the tree $c_n\circ\ldots\circ c_0(\tau)$ with its induced decoration. If $c$ is a composition data on $\tau$ and $c'$ is a composition data on $c_*\tau$, then we denote by $c'\circ c$ the composition data on $\tau$ obtained by concatenation.

We can now give the main definition of this subsection.

\begin{defi}
A \emph{bioperad} is the data of
\begin{itemize}
\item A set of objects $\on{Ob}(\oper{M})$.
\item An $\on{Ob}(\oper{M})$-multigraph $\oper{M}$.
\item For each decorated tree $\tau$ and each composition data $c$ on $\tau$, a functor
\[\chi_c:\oper{M}(\tau)\to \oper{M}(c_*(\tau))\]
\item For each pair of composition data $c$, $c'$ on $\tau$ with same target, a natural isomorphism
\[\kappa_{c,c'}:\chi_c\implies \chi_{c'}\]
\end{itemize}
such that 
\begin{itemize}
\item If $c$ and $c'$ are two composable composition data, we have
\[\chi_{c\circ c'}=\chi_c\circ\chi_{c'}\]
\item If $c,c',c''$ are three composition data on the same decorated tree with same target, then
\[\kappa_{c',c''}\circ\kappa_{c,c'}=\kappa_{c,c''}\]
\end{itemize}
\end{defi}

\begin{rem}
The fact that $\chi_{c\circ c'}=\chi_c\circ\chi_{c'}$ implies that the maps $\chi$ are entirely determined by their value on isomorphisms faces and degeneracies.
\end{rem}

\begin{rem}
The usual definition of a bicategory, involves a ``minimal'' set of coherence isomorphisms. Here, we have chosen the opposite approach of giving a maximal set of such isomorphisms. The advantage our approach is that there is no complicated coherence theorem to prove since any two compositions that we may wish to compare are related by a given isomorphism. The disadvantage is of course that it is potentially hard to prove that something is a bioperad. The approach of having a minimal set of coherence isomorphisms is sketched in \cite{toenoperations} in the one object case.
\end{rem}

There is an obvious notion of a strict morphism between bioperads but, as expected, most morphisms occurring in nature are not strict morphisms. We now define the notion of a pseudo-functor between bioperads.

\begin{defi}
A \emph{pseudo-functor} $f:\oper{M}\to\oper{N}$ is a morphism of graphs together with the data of isomorphisms
\[\upsilon_{c}:f\circ\chi^{\oper{M}}_c\to \chi_c^{\oper{N}}\circ f\]
for any composition data $c$ on $\tau$.

Those isomorphisms are such that for any pair $(c,d)$ and $(c',d')$ of composable composition data with same source and target, the following diagram commutes
\[\xymatrixcolsep{3pc}
\xymatrix{
 &\chi_{c\circ d}\circ f\ar[r]^{\kappa^{\oper{N}}_{c\circ d,c'\circ d'}\circ f}&\chi_{c'\circ d'}\circ f\\
\chi_c\circ f\circ\chi_d\ar[ur]^{\chi_c\circ\upsilon_d}& & & \chi_{c'}\circ f\circ\chi_{d'}\ar[ul]_{\chi_{c'}\circ\upsilon_{d'}}\\
 &f\circ\chi_{c\circ d}\ar[ul]^{\upsilon_c\circ\chi_d}\ar[r]_{f\circ\kappa^{\oper{M}}_{c\circ d,c'\circ d'}}&f\circ\chi_{c'\circ d'}\ar[ur]_{\upsilon_c'\circ\chi_{d'}}
}
\]
\end{defi}

Now we explain how a bioperad can be strictified into an equivalent $2$-operad. In fact, our definition of a bioperad has so much data that it is almost tautological. We give ourselves, on each tree $\tau$, a composition data  with target the corolla with the same number of leaves. We call it the standard composition data on $\tau$.

\begin{cons}\label{strictification}
Let $\oper{M}$ be a bioperad. We define $\on{Str}\oper{M}$ to be the multi-graph whose objects are those of $\oper{M}$, and with 
\[\on{Ob}(\on{Str}\oper{M}(\{x_i\}_{i\in I};y))=\coprod_{\tau,L(\tau)\cong I, R(\tau)=y}\on{Ob}\oper{M}(\tau)\]
If $p,q\in \on{Ob}(\on{Str}\oper{M}(\{x_i\}_{i\in I};y))$ we define the set of $2$-morphisms between them to be the set of $2$-morphisms between $\chi_c(p)$ and $\chi_c(q)$ where $c$ is the standard composition data on $\tau$.

Vertical composition of $2$-morphisms is straightforward. For horizontal composition we use the isomorphisms $\kappa_{c,c'}$ to compare the potentially non-standard composition data with the standard one.
\end{cons}

If $f:\oper{M}\to\oper{N}$ is a pseudo-functor with $\oper{N}$ strict, we get a strict functor
\[\on{Str}(f):\on{Str}(\oper{M})\to\oper{N}\]
which coincides with $f$ on objects and which sends $p\in \oper{M}(\tau)$ to the composition of $f(p)$. Moreover we have a strict functor $\on{Str}\oper{M}\to\oper{M}$ which is the identity on objects and which sends $p\in\oper{M}(\tau)$ to its standard composition. It is straightforward to check that this morphism is an equivalence. This means that we have turned the morphism
\[f:\oper{M}\to\oper{N}\]
into a zig-zag of strict arrows
\[\oper{M}\leftarrow\on{Str}\oper{M}\goto{\on{Str}f}\oper{N}\]
where the left pointing arrow is a biequivalence (bijective on objects and induces an equivalence on each Hom category). 

\begin{rem}
The data of a bioperad contains the data of a pseudo-functor $\Omega\op\to \cat{Cat}$. This pseudo-functor can be strictified to an actual functor using a classical result. In the end, we get a functor $\Omega\op\to\cat{Cat}$ which satisfies Segal's condition. In particular, if we apply the nerve functor to each Hom category, we get a Segal operad. Segal operad can then be strictified to actual operads in $\S$. This approach gives an equivalent model for the strictification.
\end{rem}

\subsection{The Morita bioperad}

It is a classical fact that, if $\cat{V}$ is a monoidal category, there is a bicategory whose objects are associative algebras in $\cat{V}$ and whose morphisms are bimodules. The composition being given by relative tensor product of bimodules. We call this bicategory the \emph{Morita bicategory}. Our purpose in this subsection is to extend this structure to that of a bioperad when $\cat{V}$ is symmetric monoidal.

Let $\cat{V}$ be a cocomplete closed symmetric monoidal category and let $S$ be a class of associative algebra such that
\begin{itemize}
\item If $\{A_i\}_{i\in I}$ is a finite collection of elements of $S$ then $\otimes_iA_i$ (with its induced associative algebra structure) is in $S$.
\item If $A$ is in $S$, $A\op$ is in $S$.
\end{itemize} 

We construct a large multigraph $\oMor^d(\cat{V},S)$. Its objects are the elements of $S$. If $\{A_i\}_{i\in I}$ and $B$ are objects, we set
\[\oMor^d(\cat{V},S)(\{A_i\};B)={}_B\Mod_{\{A_i\}_I}\]

The category ${}_B\Mod_{\{A_i\}_I}$ is the category of right modules over $B\op\otimes(\otimes_IA_i)$. 

\begin{theo}
Let $(\cat{V},\otimes)$ be a symmetric monoidal category, then $\oMor^d(\cat{V},S)$ can be extended to a bioperad.
\end{theo}

\begin{proof}
(Sketch) For a family of modules $\{M_i\}_{i\in I}$ and a module $N$, we denote
\[\cat{V}(\{M_i\},N)\]
the set of maps $\otimes M_i\to N$. The axioms of a symmetric monoidal category tell us that any two interpretation of $\otimes M_i$ can be compared by a canonical isomorphism.

It suffices to construct the maps $\chi_c$ for $c$ a face or a degeneracy.

Given $M$ in ${}_{B_k}\Mod_{\{A_i\}_I}$ and $N$ in ${}_C\Mod_{\{B_j\}_J}$ the tensor product $N\otimes_{B_k}M$ is an object of ${}_C\Mod_{\{A_i\}_I\sqcup \{B_j\}_{j\in J-k}}$ which is defined to be the composition along the face which collapses the edge with label $B_k$. There is also the unit $A\in {}_A\Mod_A$ which allows us to define $\chi_c$ for degeneracies.

To construct the map $\kappa_c$, let us first look at an example. Consider the following decorated tree $\tau$ in $\oMor^d(\cat{V},S)$
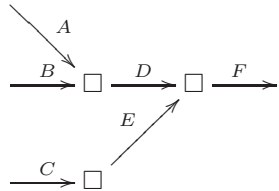
\[
\xymatrix{
 \ar[dr]^A& & & \\
 \ar[r]^B& \square\ar[r]^D&\square\ar[r]^F& \\
 \ar[r]^C&\square\ar[ur]^E& & 
 }
\]

For a composition data $c$, $\chi_c$ takes as input three elements of $\cat{V}$ $P$, $Q$ and $R$ where $P$ has a right action of $A$ and $B$ and a left action of $D$, $Q$ has a right action of $C$ and a left action of $E$ and $R$ has a right action of $D$ and $E$ and a left action of $F$. Then consider any composition data $c$ on $\tau$ which collapses the two internal edges i.e a composition data with target:
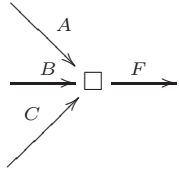
\[
\xymatrix{
 \ar[dr]^A& & \\
 \ar[r]^B&\square\ar[r]^F& \\
 \ar[ur]^C& &
 }
\]

Then $\chi_c(P,Q,R)$ is an object of $\cat{V}$ such that the functor $\cat{V}(\chi_c(P,Q,R),-)$ is isomorphic to the sub functor of $\cat{V}(P,Q,R;-)$ sending $N$ to those maps $P\otimes Q\otimes R\to N$ intertwining the $D$ and $E$ actions. For instance if $\cat{V}$ is the category of modules over a ring, then $\cat{V}(\chi_c(P,Q,R),N)$ would be the set of trilinear maps $f:P\times Q\times R\to N$ with the additional property that $f(d.p,q,r)=f(p,q,r.d)$ and $f(p,e.q,r)=f(p,q,r.e)$. 

In general, we observe that if $c$ and $c'$ are two composition data with same source and target, the functor represented by $\chi_c(P_i)$ and $\chi_{c'}(P_i)$ are isomorphic. Thus, by Yoneda's lemma, the map $\kappa_{c,c'}$ is uniquely determined and therefore must satisfy the required properties.
\end{proof}

\begin{rem}
It is proved in \cite{shulmanconstructing} that there is actually a symmetric monoidal bicategory which extends the Morita bicategory of $\cat{V}$ whenever $\cat{V}$ is symmetric monoidal.

In fact a symmetric monoidal bicategory is a richer structure than a bioperad. We believe that in general, a symmetric monoidal bicategory has an underlying bioperad, just as a symmetric monoidal category has an underlying operad.
\end{rem}

\subsection{The Morita bioperad of a symmetric monoidal model category}

\begin{defi}\label{admissible class of algebras}
Let $(\cat{V},\otimes,\un)$ be a symmetric monoidal cofibrantly generated model category. We say that an associative algebras in $\cat{V}$ is admissible if the model structure on $\Mod_A$ transferred along the functor
\[\Mod_A\to\cat{V}\]
exists and moreover, the forgetful functor preserves cofibrations.

We say that a class of associative algebras in $\cat{V}$ is admissible if all its members are admissible and it is stable under taking tensor products and opposite algebras.
\end{defi}

\begin{cons}\label{Morita bioperad}
Let $S$ be an admissible class of algebras. We construct a large bioperad $\oMor(\cat{V},S)$. Its objects are the elements of $S$. If $\{A_i\}_{i\in I}$ and $B$ are in $S$, we define
\[\oMor(\{A_i\};B)=w({}_B\Mod_{\{A_i\}_I})_c\]

The category ${}_B\Mod_{\{A_i\}_I}$ is the category of right modules over $B\op\otimes(\otimes_IA_i)$. By assumption, it has a model structure transferred from the model structure on $\cat{V}$ and $\oMor(\{A_i\};B)$ is the subcategory of weak equivalences between cofibrant objects in that model category. The bioperad structure follows almost directly from that on $\oMor^d(\cat{V},S)$. The only thing that needs to be checked is that the tensor product of bimodules preserves cofibrant objects and weak equivalences between them. But this fact follows from the next proposition.

When $\cat{V}$ is such that the class all associative algebras of $\cat{V}$ are admissible, we write $\oMor(\cat{V})$ for the Morita bioperad with respect to this class.
\end{cons}

\begin{prop}\label{tensor product is left Quillen bifunctor}
Let $A$, $B$ and $C$ be three associative algebras in $S$. The relative tensor product
\[-\otimes_B-:{}_A\Mod_B\times {}_B\Mod_C\to {}_A\Mod_C\]
is a Quillen bifunctor.
\end{prop}

\begin{proof}
It suffices to check it on generators.
\end{proof}

\subsection{Functoriality of the Morita bioperad}

\begin{prop}\label{Morita functoriality}
Let $F:\cat{V}\to\cat{W}$ be a lax symmetric monoidal left Quillen functor. Let $S$ and $T$ be two admissible classes of associative algebras in $\cat{V}$ and in $\cat{W}$ such that $F(S)\subset T$. Then $F$ induces a pseudo-functor
\[\oMor(F):\oMor(\cat{V},S)\to\oMor(\cat{W},S)\]
\end{prop}
 
\begin{proof}
It suffices to notice that $F$ induces a left Quillen functor
\[{}_A\Mod_B\to {}_{F(A)}\Mod_{F(B)}\]
for any pair of algebras $A$ and $B$ in $S$. Moreover if $A$ $B$ and $C$ are three algebras in $S$, the diagram
\[\xymatrix{
{}_A\Mod_B\times {}_B\Mod_C\ar[r]^{-\otimes_B-}\ar[d]_F& {}_A\Mod_C\ar[d]^F\\
{}_{F(A)}\Mod_{F(B)}\times {}_{F(B)}\Mod_{F(C)}\ar[r]_{\hspace{30pt}-\otimes_{FB}-}& {}_{F(A)}\Mod_{F(C)}
}
\]
commutes up to a unique isomorphism.
\end{proof}

\begin{prop}\label{Morita equivalence}
Let $F:\cat{V}\to\cat{W}$ be a lax symmetric monoidal left Quillen equivalence. Let $S$ and $T$ be as above. Then $F$ induces a homotopically fully faithful pseudo-functor
\[\oMor(F):\oMor(\cat{V},S)\to\oMor(\cat{W},T)\]
Assume moreover, that for any algebra $A$ in $T$ there is a weak equivalence of associative alegbras $F(B)\to A$ for some associative algebra $B\in S$. Then the map $\oMor(F)$ is a weak equivalence.
\end{prop} 

\begin{proof}
For the first claim, it suffices to observe that $F$ induces a left Quillen equivalence
\[{}_A\Mod_B\to {}_{F(A)}\Mod_{F(B)}\]

Then we claim that if $G:\cat{X}\to\cat{Y}$ is a left Quillen equivalence, the induced map
\[N(w\cat{X}_c)\to N(w\cat{Y}_c)\]
is a weak equivalence.

The second claim is an essential surjectivity claim. It suffices to check that being isomorphic in $\on{Ho}(\Alg(\cat{C}))$ implies being isomorphic in $\on{Ho}(\oMor(\cat{C}))$. Let $u:A\to B$ be an equivalence of associative algebras. Let $B^m$ be a cofibrant $A$-$B$-bimodule which is weakly equivalent to $B$ seen as an $A$-$B$-bimodule through the map $u$. Then, for any associative algebra $C$, we get a map
\[B^m\otimes_B-:{}_B\Mod_C\to{}_A\Mod_C\]
It is straightforward to check that it is a Quillen equivalence. Thus the algebras $A$ and $B$ are isomorphic in $\on{Ho}(\oMor(\cat{C}))$.
\end{proof}

\section{The operad of model categories}

In this section, we construct a large operad $\oModCat$ whose objects are model categories and whose category of morphisms is the category of left Quillen multi-functors and equivalences between them. We construct a map from this operad to the operad of complete Segal spaces which coincides with Rezk's nerve construction on objects.

\subsection{Simplicial operad of model categories}

\begin{defi}
Let $\cat{X}$ and $\cat{Y}$ be two model categories. Let $F$ and $G$ be two left Quillen functors $\cat{X}\to\cat{Y}$. A \emph{natural weak equivalence} $\alpha:F\to G$ is a natural transformation $F_{|\cat{X}_c}\to G_{|\cat{X}_c}$ which is objectwise a weak equivalence.
\end{defi}

There is an obvious (vertical) composition between natural weak equivalences but there is also an horizontal composition between natural transformation which preserves natural weak equivalences by the following proposition.

\begin{prop}
Let $\cat{X}$, $\cat{Y}$ and $\cat{Z}$ be three model categories and let $F$, $G$ be two left Quillen functors from $\cat{X}$ to $\cat{Y}$ and $K$ and $L$ be two left Quillen functors from $\cat{Y}\to\cat{Z}$. Let $\alpha$ be a natural weak equivalence between $F$ and $G$ and $\beta$ be a natural weak equivalence between $K$ and $L$, then the horizontal composition of $\alpha$ and $\beta$ is a natural weak equivalence.
\end{prop}

\begin{proof}
The horizontal composition evaluated at a cofibrant object $x$ is the composition
\[KF(x)\goto{\beta F}LF(x)\goto{L\alpha}LG(x)\]
Since $F$ is left Quillen, $F(x)$ is cofibrant and the first map is a weak equivalence. The second map is $L$ applied to $\alpha(x):F(x)\to G(x)$ which is a weak equivalence between cofibrant objects. Since $L$ is left Quillen, this is an equivalence as well.
\end{proof}

\begin{cons}
The $2$-category $\cat{ModCat}$ has as objects the model categories and its category of morphism from $\cat{X}$ to $\cat{Y}$ is the category whose objects are left Quillen functors: $\cat{X}\to\cat{Y}$ and morphisms are natural weak equivalences between left Quillen functors.
\end{cons}

Now we want to extend $\cat{ModCat}$ to an operad.

For two model categories $\cat{X}$ and $\cat{Y}$, one can put a product model structure on $\cat{X}\times\cat{Y}$, but the left Quillen functors from $\cat{X}\times\cat{Y}$ to $\cat{Z}$ are usually not the right thing to consider. The correct notion of ``pairing'' $\cat{X}\times\cat{Y}\to \cat{Z}$ is the notion of a left Quillen bifunctor.

We need a version of a Quillen multifunctor with more than two inputs. Let us first recall the definition of the cube category.

\begin{defi}
The \emph{$n$-dimensional cube} is the poset of subsets of $\{1,\ldots,n\}$. We use the notation $\cat{P}(n)$ to denote that category. Equivalently, $\cat{P}(n)$ is the product of $n$ copies of $\cat{P}(1)=[1]$. The category $\cat{P}_{1}(n)$ is the full subcategory of $\cat{P}(n)$ contatining all objects except the maximal element.
\end{defi}

\begin{defi}\label{hyper pushout product axiom}
If $(\cat{X}_i)_{i\in\{1,\ldots,n\}}$ is a family of categories and $f_i$ is an arrow in $\cat{X}_i$ for each $i$, we denote by $C(f_1,\ldots,f_n)$ the product
\[\prod_i f_i:\cat{P}(n)\to\prod_i{\cat{X}_i}\]
\end{defi}

\begin{defi}
Let $(\cat{X}_i)_{i\in\{1,\ldots,n\}}$ and $\cat{Y}$ be model categories. Let $T:\prod_{i=1}^n\cat{X}_i\to \cat{Y}$ be a functor. We say that $T$ is a \emph{left Quillen $n$-functor} if it satisfies the following three conditions
\begin{itemize}
\item If we fix all variables but one. The induced functor $\cat{X}_i\to\cat{Y}$ is a left adjoint.
\item If $f_i:A_i\to B_i$ is a cofibration in $\cat{X}_i$ for $i\in\{1,\ldots,n\}$ then the map
\[\on{colim}_{\cat{P}_1(n)}T(C(f_1,\ldots,f_n))\to T(B_1,\ldots,B_n)\]
is a cofibration in $\cat{Y}$
\item If further one of the $f_i$ is a trivial cofibration, then the map
\[\on{colim}_{\cat{P}_1(n)}T(C(f_1,\ldots,f_n))\to T(B_1,\ldots,B_n)\]
is a trivial cofibration in $\cat{Y}$
\end{itemize}

A \emph{natural weak equivalence between left Quillen $n$-functors} $T$ and $S$ is a natural transformation $T_{|\prod{\cat{X}_i}_c}\to T'_{|\prod{\cat{X}_i}_c}$ which is objectwise a weak equivalence.
\end{defi}

\begin{cons}\label{construction of oModCat}
We can now construct an operad in $\Cat$ denoted $\oModCat$ whose objects are model category and whose category of operations $\oModCat(\{\cat{X}_i\};\cat{Y})$ is the category of left Quillen $n$-functors $\prod_i\cat{X}_i\to\cat{Y}$ and natural weak equivalences. As usual, we also denote $\oModCat$ the operad in $\S$ obtained by applying the nerve functor to each category of multi-morphisms.
\end{cons}

\bigskip

Now, take $\cat{V}$ to be a cofibrantly generated closed symmetric monoidal model category. Recall that for a class of associative algebras $S$ in $\cat{V}$, $\oMor(\cat{V},S)$ is the Morita operad of $\cat{V}$ on the elements of $S$ (see \ref{Morita bioperad} for a construction).

\begin{prop}\label{from rings to cat}
There is a map of bioperads
\[\oMor(\cat{V},S)\to\oModCat\]
which sends the object $A$ to the model category  $L\Mod_A$ and sends an element $M$ in $({}_B\Mod_{\{A_i\}})_c$ to the left Quillen multi-functor
\[ M\otimes_{\otimes_iA_i}(\otimes_i-):\prod_iL\Mod_{A_i}\to L\Mod_B\]

\end{prop}

\begin{proof}
Straightforward.
\end{proof}

\subsection{Model categories and complete Segal spaces}

The goal of this subsection and the following is to relate the simplicial operad $\oModCat$ to a reasonable model of the simplicial operad of $\infty$-categories.

Rezk defines in \cite{rezkmodel} a model structure on simplicial spaces whose fibrant objects are Reedy fibrant simplicial spaces satisfying Segal condition and a completeness condition. It has been shown that it is a reasonable model for $\infty$-categories. There are explicit Quillen equivalences between the model category and Joyal model structure on simplicial sets and Bergner model structure on simplicial categories. In this paper, we take the convention of calling complete Segal space any simplicial space which is levelwise equivalent to a complete Segal space in Rezk's sense. We use the terminology \emph{fibrant complete Segal space} for what Rezk calls a complete Segal space.

The category of simplicial spaces with Rezk's model structure is a simplicial symmetric monoidal (for the cartesian product) cofibrantly generated model category. Hence, there is a simplicial operad $\oper{CSS}$ whose objects are fibrant complete Segal spaces and with
\[\oper{CSS}(\{X_i\};Y)=\Map_{s\S}(\prod_iX_i,Y)\]

We will use the operad $\oper{CSS}$ as our model for the correct operad of $\infty$-categories. Our goal is to construct a map
\[\oModCat\to\oper{CSS}\]
Our first task is to define this map on objects. For any relative category $(C,wC)$, there is a simplicial space $N(C,wC)$ constructed in \cite{rezkmodel} whose space of $n$-simplices is the nerve of the category of weak equivalences in the relative category $C^{[n]}$. In particular, we can apply this nerve to a model category. It has been proved by Barwick and Kan (see \cite{barwickpartial}) that the resulting simplicial space is a complete Segal space. Unfortunately, the assignment $\cat{M}\mapsto N(\cat{M},w\cat{M})$ is not functorial with respect to left Quillen functors. In order to remedy this we use the following observation.

\begin{prop}
Assume that $M$ has a functorial cofibrant replacement functor. Then, the full inclusion of relative categories $(\cat{M}_c,w\cat{M}_c)\to(\cat{M},w\cat{M})$ induces a levelwise weak equivalence
\[N(\cat{M}_c,w\cat{M}_c)\to N(\cat{M},w\cat{M})\]
\end{prop}

\begin{proof}
The cofibrant replacement functor induces a functor
\[Q_n:w\cat{M}^{[n]}\to w\cat{M}_c^{[n]}\]
If $i_n$ is the inclusion $w\cat{M}_c^{[n]}\to w\cat{M}^{[n]}$, we have natural weak equivalences
\[i_n\circ Q_n\goto{\simeq} \id_{w\cat{M}^{[n]}}\]
and 
\[Q_n\circ i_n\goto{\simeq}\id_{w\cat{M}_c^{[n]}}\]

Taking the nerve, these natural transformations are turned into homotopies. Therefore $Q_n$ induces a weak equivalence from $N_n(\cat{M},w\cat{M})$ to $N_n(\cat{M}_c,w\cat{M}_c)$ which is a homotopy inverse to the inclusion
\[N_n(\cat{M}_c,w\cat{M}_c)\to N_n(\cat{M},w\cat{M})\]
\end{proof}

By \cite{barwickpartial}, this implies that $N(\cat{M}_c,w\cat{M}_c)$ is a complete Segal space which is also a model for the $\infty$-category presented by $\cat{M}$ and which is functorial in left Quillen functors.

\subsection{The simplicial category of relative categories}

We want to show that the category of relative categories is enriched in spaces. For $C$ and $D$ two relative categories, the space of maps between them is the nerve of the category of natural transformations between functors $C\to D$ which are objectwise weak equivalences. Our goal in this subsection is to show that this mapping space is also the space of maps between $N(C,wC)$ and $N(D,wD)$.

It will be convenient to be a bit more general and consider the category $\cat{DCat}$ of double categories. A double category is a category object in the category of categories. In concrete term, a double category $\mathbb{D}$ is the data of
\begin{itemize}
\item a set of objects $\on{Ob}(\mathbb{D})$,
\item two categories $\mathbb{D}_h$ and $\mathbb{D}_v$ whose set of objects are $\on{Ob}(\mathbb{D})$,
\item a set of $2$-squares (i.e. squares whose horizontal arrows are in $\mathbb{D}_h$ and vertical arrows are in $\mathbb{D}_v$.
\end{itemize}
Moreover, the square can be composed either vertically or horizontally and there are associativity and unitality axioms.

There are two obvious functors from double categories to categories namely $\mathbb{D}\mapsto \mathbb{D}_h$ and $\mathbb{D}\mapsto \mathbb{D}_v$. These two functors both have left adjoints that we denote $C\mapsto C^h$ and $C\mapsto C^v$. The double category $C^h$ is the double category whose horizontal category is $C$, vertical category is the discrete category on the objects of $C$ and $2$-squares are the obvious ones. The category $C^v$ is defined symmetrically.

Given a category $C$ together with two wide subcategories $H$ and $V$, there is a double category $\mathbb{D}(C,H,V)$ whose vertical category is $V$, horizontal category is $H$ and admissible $2$-squares are commutative squares in $C$. For instance, the category $C^h$ is the category $\mathbb{D}(C,C,C_0)$ (where $C_0$ denotes the subcategory of $C$ which contains only the identities) and $C^v$ is $\mathbb{D}(C,C_0,C)$.

There is a functor $\cat{RelCat}\to\cat{DCat}$ sending $(C,wC)$ to $\mathbb{D}(C,C,wC)$.

If $C$ and $D$ are two categories, we define $C\square D$ the double category $\mathbb{D}(C\times D,C\times D_0,C_0\times D)$.

\medskip

There is a fully faithful nerve functor $N:\cat{DCat}\to s\S$ sending $\mathbb{D}$ to the bisimplicial set whose $[p,q]$ simplices are the functors $[p]\square [q]\to \mathbb{D}$. If we break the symmetry and see a double category as a category in categories, we see that applying the usual nerve once sends a double category to a simplicial object in categories. Applying the nerve degreewise yields a bisimplicial set which is exactly the one we are describing here.

Let us compute $N([p]\square[q])$. It is the bisimplicial set whose $[m,n]$ simplices are the functors
\[[m]\square[n]\to [p]\square[q]\]

This is isomorphic to the set of  pairs of functors $([m]\to[p],[n]\to[q])$. Hence $N([p]\square[q])\cong \Delta[p,q]$

This implies that the double nerve of the double category $\mathbb{D}(C,C,wC)$ coincides with Rezk's classifying functor for relative categories. Following Rezk, we write $N(C,wC)$ instead of $N(\mathbb{D}(C,C,wC))$.

\begin{prop}
There is a natural isomorphism $N(\mathbb{D})\times N(\mathbb{D}')\cong N(\mathbb{D}\times\mathbb{D}')$.
\end{prop}

\begin{proof}
The right hand side, in degree $[p,q]$, is the set of functor $[p]\square[q]\to \mathbb{D}\times\mathbb{D}'$. The left hand side, in degree $[p,q]$, is the product of the set of functors $[p]\square[q]\to \mathbb{D}$ with the set of functors $[p]\square[q]\to\mathbb{D}'$.
\end{proof}

Let $\mathbb{D}$ and $\mathbb{E}$ be two double categories, there is an internal Hom $\mathbb{E}^\mathbb{D}$ satisfying the universal property
\[\cat{DCat}(\mathbb{C}\times\mathbb{D},\mathbb{E})\cong\cat{DCat}(\mathbb{C},\mathbb{E}^{\mathbb{D}})\]

\begin{prop}
There is an isomorphism $N(\mathbb{E}^\mathbb{D})\cong N\mathbb{E}^{N\mathbb{D}}$
\end{prop}

\begin{proof}
A $[p,q]$ simplex in the left hand side is a functor 
\[[p]\square[q]\to \mathbb{E}^\mathbb{D}\]

By adjunction, this is a functor
\[([p]\square[q])\times\mathbb{D}\to \mathbb{E}\]

On the other hand, a $[p,q]$ simplex of the right hand side is a map 
\[\Delta[p,q]\to N\mathbb{E}^{N\mathbb{D}}\]

which by adjunction is the same data as a map
\[\Delta[p,q]\times N\mathbb{D}\to N\mathbb{E}\]

But since $\Delta[p,q]\cong N([p]\square [q])$ and $N$ is fully faithful and product preserving, this set is isomorphic to the set of functors
\[([p]\square[q])\times\mathbb{D}\to \mathbb{E}\]
\end{proof}

\begin{rem}
Assume that $\mathbb{E}=\mathbb{D}(E,E,wE)$ and $\mathbb{D}=\mathbb{D}(D,D,wD)$, then $\mathbb{E}^\mathbb{D}$ is the double category $\mathbb{D}(E^D,E^D,wE^D)$ where $E^D$ denotes the category of weak equivalences preserving functors and natural transformations and $wE^D$ is the wide subcategory whose arrows are natural transformations which are objectwise weak equivalences. In other words, the functor $\cat{RelCat}\to\cat{DCat}$ is not only fully faithful but it also preserves products and inner Homs.
\end{rem}

There are two projections $\Delta^2\to\Delta$ that we denote $p_h$ and $p_v$. By precomposition, they induce two maps $p_h^*$ and $p_v^*$ $:\S\to s\S$, We choose the labels so that if $C$ is a category, we have 
\[p_h^*NC\cong N(C^h)\;\;\; p_v^*NC\cong N(C^v)\]

We construct a mapping space $v\Map_{s\S}$ in $s\S$ by declaring that the $[q]$ simplices of $v\Map_{s\S}(X,Y)$ are the maps
\[X\times p_v^*[q]\to Y\]

We define $h\Map_{s\S}$ analogously.

\begin{prop}
Let $E$ and $D$ be two relative categories, then 
\[v\Map_{s\S}(N(D,wD),N(E,wE))\cong N(wE^D)\]
\[h\Map_{s\S}(N(D,wD),N(E,wE))\cong N(E^D)\]
\end{prop}

\begin{proof}
Indeed by the above proposition and remark, we have
\[v\Map_{s\S}(ND,NE)_q:= s\S(ND\times p_v^*(\Delta[q]), NE)\cong s\S(p_v^*(\Delta[q]),N(E^D))\]
Now, by construction, $p_v^*(\Delta[q])\cong N([q]^v)$. Therefore 
\[v\Map_{s\S}(ND,NE)_q\cong \cat{DCat}([q]^v,E^D)\cong N(wE^D)_q\]
\end{proof}

Note that $p_v^*[q]=N([q]^v)$ is the bisimplicial set whose $[m,n]$ simplices are functors $[n]\to [q]$. In other words, $p_v^*[q]$ is the simplicial set $\Delta[q]$ seen as a constant simplicial space. Hence the space $v\Map_{s\S}$ is what Rezk denotes $\Map_{s\S}$.

To conclude, we have proved the following

\begin{prop}
The relative nerve functor from relative categories to simplicial spaces can be promoted to a fully faithful functor of simplicially enriched categories if we give $\cat{RelCat}$ the following simplicial enrichment
\[\Map_{\cat{RelCat}}(D,E)=N(wE^D)\]
$\hfill\square$
\end{prop}

\subsection{Comparison with the homotopy correct mapping space}

For general relative categories $D$ and $E$, the space $\Map_{\cat{RelCat}}(D,E)$ is usually not equivalent to the derived mapping space. However, if $E$ is a model category, we have the following proposition.

\begin{prop}\label{homotopy correct mapping space}
If $\cat{E}$ is a cofibrantly generated model category which is Quillen equivalent to a combinatorial model category and $D$ is any relative category, then, the map
\[\Map_{\cat{RelCat}}(D,\cat{E})=\Map_{s\S}(N(D,w),N(\cat{E},w))\to\Map_{s\S}(N(D,w),RN(\cat{E},w))\]
is an equivalence for any fibrant replacement $N(\cat{E},w)\to RN(\cat{E},w)$.
\end{prop}

\begin{proof}
(1) Assume first that $D$ is an ordinary category (i.e. $wD\subset \on{iso}(D)$), then $\cat{E}^D$ can be given the projective model structure and by classical model category method, we see that $\Map_{\cat{RelCat}}(D,\cat{E})$ does not change if we replace $\cat{E}$ by a Quillen equivalent model category. The same is of course true for $\Map_{s\S}(N(D,w),RN(\cat{E},w))$ and we can assume by a result of Dugger (see \cite{duggerreplacing}) that $\cat{E}$ is a combinatorial simplicial model category. In that case, the result is proved by Lurie in \cite[Proposition A.3.4.13]{lurietopos}.

(2) Now assume that $D$ is a general relative category, then $\Map_{\cat{RelCat}}(D,\cat{E})$ fits in the following pullback diagram
\[
\xymatrix{
\Map(N(D,wD),\cat{E})\ar[d]\ar[r]&\Map(N(D,\on{iso}D),\cat{E})\ar[d]\\
N_1(w\cat{E},w\cat{E})^{\on{Ar}(wD)}\ar[r]&N_1(\cat{E},w\cat{E})^{\on{Ar}(wD)}
}
\]
where $\on{Ar}(wD)$ denotes the set of arrows of $wD$. In words, this is saying that the space of maps $(D,wD)\to \cat{E}$ is the space of maps $C\to \cat{E}$ with the property that the restriction to $wC$ lands in the space of weak equivalences of $\cat{E}$. Since a model category is saturated, the map $N_1(w\cat{E},w\cat{E})\to N_1(\cat{E},w\cat{E})$ coincides with the inclusion
\[N(\cat{E},w\cat{E})_{\on{hoequiv}}\to N_1(\cat{E},w\cat{E})\]
and in particular is an inclusion of connected components and hence is a fibration. This means that the above pullback square is a homotopy pullback. On the other hand, we have a similar pullback diagram for $Y=RN(\cat{E},w\cat{E})$:
\[
\xymatrix{
\Map_{s\S}(N(D,wD),Y)\ar[d]\ar[r]&\Map_{s\S}(N(D,\on{iso}D),Y)\ar[d]\\
Y_{\on{hoequiv}}^{\on{Ar}(wD)}\ar[r]&Y_1^{\on{Ar}(wD)}
}
\]
which is a homotopy pullback square for the same reason. 

(3) There is a comparison map from the first square to the second square. In order to prove that the component on the upper-left corner is an equivalence, it suffices to show that the other three components are equivalences. The case of the upper-right corner is dealt with in the first paragraph. To deal with the other two cases, first, we recall that $N(\cat{E},w\cat{E})$ is a complete Segal space. In particular, the map $N(\cat{E},w\cat{E})\to Y$ is a levelwise equivalence. This means that the bottom right corner is an equivalence. The completeness also implies that 
\[N(\cat{E},w\cat{E})_{\on{hoequiv}}\to Y_{\on{hoequiv}}\]
is an equivalence.
\end{proof}

For $\cat{M}$, a model category, we can form the simplicial space $N(\cat{M}_c,w\cat{M}_c)$. Note that a left Quillen functor between model categories $\cat{M}\to\cat{N}$ induces a map
\[N(\cat{M}_c,w\cat{M}_c)\to N(\cat{N}_c,w\cat{N}_c)\]
and more generally, a left Quillen multi-functor $\prod_i\cat{M}_i\to\cat{N}$ produces a map
\[\prod_iN((\cat{M}_i)_c,(w\cat{M}_c)_i)\to N(\cat{N}_c,w\cat{N}_c)\]
Hence sending a model category to its full subcategory of cofibrant objects, we get a map
\[\oModCat\to \oper{R}el\oper{C}at^{\on{naive}}\]
where $\oper{R}el\oper{C}at^{\on{naive}}$ denotes the simplicial operad whose objects are relative categories and with 
\[\oper{R}el\oper{C}at^{\on{naive}}(\{C_i\};D)=\Map_{\cat{RelCat}}(\prod_iC_i,D)\]

This map can be composed with Rezk's nerve to get a map
\[\phi:\oModCat\to\oper{CSS}^{\on{naive}}\]
where $\oper{CSS}^{\on{naive}}$ denotes the simplicial operad whose objects are (non-fibrant) complete Segal spaces and morphisms are given by the (underived) mapping spaces. The operad $\oper{CSS}$ sits inside $\oper{CSS}^{\on{naive}}$ as the full suboperad on fibrant objects. 

We say that a complete Segal space $Y$ is good if for any relative category $(C,wC)$, the map
\[\Map_{s\S}(N(C,wC),Y)\to\Map_{s\S}(N(C,wC),RY)\]
is an equivalence for any Reedy fibrant replacement $Y\to RY$. Good simplicial spaces are stable under product, hence we can consider the full suboperad $\oper{C}at_{\infty}\hookrightarrow \oper{CSS}^{\on{naive}}$ on good simplicial spaces. The inclusion
\[\oper{CSS}\to\oper{C}at_{\infty}\]
is fully faithful by definition and essentially surjective (indeed any good simplicial space is equivalent in the simplicial operad $\oper{C}at_\infty$ to a Reedy fibrant replacement). To conclude, we have proved the following theorem.

\begin{theo}\label{comparison model category complete Segal spaces}
Let $\oModCat^c$ be the full suboperad of $\oModCat$ on model categories that are Quillen equivalent to a combinatorial model category.

The assignment $\cat{M}\mapsto N(\cat{M}_c,w\cat{M}_c)$ extends to a map of simplicial operad
\[\oModCat^c\to\oper{C}at_{\infty}\]
where $\oper{C}at_{\infty}$ is a simplicial operad which is weakly equivalent to the operad of fibrant complete Segal spaces. 
\end{theo}

\begin{rem}
Let $\oper{P}r^L$ be the suboperad of $\oper{C}at_\infty$ whose objects are good complete Segal spaces with all homotopy colimits whose underlying quasi-category is presentable and whose morphisms are given by multi-functors which preserve the colimits in each variable. Then the map
\[\oModCat^c\to\oper{C}at_\infty\]
factors through $\oper{P}r^L$ and we conjecture that it is a homotopically fully faithful map.
\end{rem}

\begin{exa}
Take $\oper{M}$ be the nonsymmetric operad freely generated by an operation in degree $0$ and $2$. An algebra over $\oper{M}$ in $\cat{Set}$ is a set with a binary multiplication and a base point. Let $\oper{P}(n)$ be the operad in $\Cat$ which is given in degree $n$ by the groupoid whose objects are points of $\oper{M}(n)$ and a with a unique morphism between any two objects. Then an algebra over $\oper{P}$ is a monoidal category.

The nerve of $\oper{P}$ is an $\oper{A}_\infty$-operad $N\oper{P}$. According to the previous theorem, if $\cat{M}$ is a monoidal model category, then $N(\cat{M}_c)$ is a $N\oper{P}$-algebra in the simplicial category of complete Segal spaces.

One could similarly show that if $\cat{M}$ is a braided (resp. symmetric) monoidal model category, then $N(\cat{M}_c)$ is an algebra over a certain $\oper{E}_2$ (resp. $\oper{E}_\infty$)-operad.

One can also work with many objects operads. If $\cat{V}$ is a monoidal model category and $\cat{M}$ is a model category left tensored over $\cat{V}$, the the pair $(N(\cat{V}_c),N(\cat{M}_c))$ is an algebra over an operad equivalent to $L\oMod$ in $\oper{C}at_{\infty}$. 
\end{exa}

\begin{rem}
Note that the category of relative categories is actually enriched in relative categories. Therefore, using Rezk's functor, we can actually form an enrichment of $\cat{RelCat}$ in simplicial space. Similarly, the $2$-operad $\oModCat$ can be promoted to an operad enriched in relative categories by allowing natural transformations between Quillen functors that are not weak equivalences. We claim that the map $\oModCat\to\oper{C}at_{\infty}$ is just the degree $0$ part of a map of operad enriched in simplicial spaces. 

The operad $\oMor(\cat{V},S)$ can also be extended into an operad enriched in relative categories by allowing any map between bimodules. The map $\oMor(\cat{V},S)\to\oModCat$ constructed in \ref{from rings to cat} could then be extended to a map of operads in simplicial space. We will not use this additional structure in this paper.
\end{rem}

\section{Modules over an $\oper{O}$-algebra}

In this section, we give ourselves a one-object operad $\oper{O}$ and we construct a family of theories of modules over $\oper{O}$-algebras. These module categories are parametrized by associative algebras in the category of right modules over $\oper{O}$. Assuming that the symmetric monoidal model category we are working with satisfies certain reasonable conditions, these categories of modules can be given a model category structure.

The reader is invited to refer to the appendices for background material about operads and model categories.

\subsection{Definition of the categories of modules}

In this subsection $(\cat{C},\otimes,\un)$ denotes a simplicial symmetric monoidal category for which $\otimes$ preserves colimits in both variables. We do not assume any kind of model structure.

\begin{defi}\label{PMod}
Let $P$ be an associative algebra in right modules over $\oper{O}$. The operad $P\oMod$ of \emph{$P$-shaped $\oper{O}$-modules} has two objects $a$ and $m$. Its spaces of operations are as follows 
\begin{align*}
P\oMod(a^{\boxplus n};a)&=\oper{O}(n)\\
P\oMod(a^{\boxplus n}\boxplus m;m)&=P(n)
\end{align*}

Any other space of operation is empty. The composition is left to the reader.
\end{defi}

Any category that can reasonably be called a category of modules over an $\oper{O}$-algebras arises in the above way as is shown by the following easy proposition:

\begin{prop}
Let $\oper{M}$ be an operad with two objects $a$ and $m$ and satisfying the following properties:
\begin{itemize}
\item $\oper{M}(*;a)$ is empty if $*$ contains the object $m$.
\item $\oper{M}(a^{\boxplus n};a)=\oper{O}(n)$
\item $\oper{M}(*;m)$ is non empty only if $*$ contains exactly one copy of $m$.
\end{itemize}
Then $\oper{M}=P\oMod$ for some $P$ in $\Mod_{\oper{O}}[\oAss]$.
\end{prop}
\begin{proof}
We define $P(n)=\oper{M}(a^{\boxplus n}\boxplus m;m)$. Using the fact that $\oper{M}$ is an operad, it is easy to prove that $P$ is an object of $\Mod_{\oper{O}}[\oAss]$ and that $\oper{M}$ coincides with $P\oMod$.
\end{proof}

We denote by $\cat{C}[P\oMod]$ the category of algebras over this two-objects operad in the category $\cat{C}$. Objects of this category are pairs $(A,M)$ of objects of $\cat{C}$. The object $A$ is an $\oper{O}$-algebra and the object $M$ has an action of $A$ parametrized by the spaces $P(n)$. Maps in this category are pairs $(f,g)$ preserving all the structure.

Note that the construction $P\mapsto P\oMod$ is a functor from $\Mod_{\oper{O}}[\oAss]$ to the category of operads. It preserves weak equivalences between objects of $\Mod_{\oper{O}}[\oAss]$. We can in fact improve this homotopy invariance.

We construct a category $\cat{OM}$. Its objects are pairs $(\oper{O},P)$ where $\oper{O}$ is a one-object operad and $P$ is an associative algebra in right modules. Its morphisms $(\oper{O},P)\to(\oper{O}',P')$ consist of a morphisms of operads $f:\oper{O}\to\oper{O}'$ together with a morphisms of associative algebras in $\oper{O}$-modules $P\to P'$ where $P$ is an seen as an $\oper{O}$-module by restriction along $f$. We say that a map in $\cat{OM}$ is a \emph{weak equivalence} if it induces a weak equivalence on $\oper{O}$ and $\oper{P}$.

\begin{prop}\label{invariance of operads}
The functor $\cat{OM}\to\cat{Oper}$ sending $(\oper{O},P)$ to $P\oMod$ preserves weak equivalences.\hfill$\square$
\end{prop}

\begin{defi}
Let $A$ be an $\oper{O}$-algebra in $\cat{C}$. The \emph{category of $P$-shaped $A$-modules} denoted by $P\Mod_A$ is the subcategory of $\cat{C}[P\oMod]$ on objects of the form $(A,M)$ and of maps of the form $(\id_A,g)$.
\end{defi}

Note that there is an obvious forgetful functor $P\Mod_A\to\cat{C}$. One easily checks that it preserves colimits and limits.

\medskip

This abstract definition recovers well-known examples. We can try to model left and right modules over associative algebras. Take $\oper{O}$ to be $\oAss$ as an operad in the category of sets. The category $\cat{Ass}$ is the category of non-commutative sets (it is defined for instance in \cite{angeltveitcyclic}). Its objects are finite sets and its morphisms are pairs $(f,\omega)$ where $f$ is a map of finite sets and $\omega$ is the data of a linear ordering of each fiber of $f$.

\begin{cons}
Let $\cat{Ass}^-$ (resp. $\cat{Ass}^+$) be the category whose objects are based finite sets and whose morphisms are  pairs $(f,\omega)$ where $f$ is a morphisms of based finite sets and $\omega$ is a linear ordering of the fibers of $f$ which is such that the base point is the smallest (resp. largest) element of the fiber over the base point of the target of $f$.

Let $R$ (resp. $L$) be the right module over $\oAss$ defined by the formulas
\begin{align*}
R(n)&=\cat{Ass}^-(\{*,1,\ldots,n\},\{*\})\\
L(n)&=\cat{Ass}^+(\{*,1,\ldots,n\},\{*\})
\end{align*}
Let us construct a pairing
\[R(n)\times R(m)\to R(n+m)\]
Note that specifying a point in $R(n)$ is equivalent to specifying a linear order of $\{1,\ldots,n\}$. Let $f$ be a point in $R(n)$ and $g$ be a point in $R(m)$. We define their product to be the map whose associated linear order of $\{1,\ldots,n+m\}$ is the linear order induced by $n$ concatenated with the linear order induced by $g$.
\end{cons}

\begin{prop}
Let $A$ be an associative algebra in $\cat{C}$. $L\Mod_A$ (resp. $R\Mod_A$) is isomorphic to the category of left (resp. right) modules over $A$. $\hfill\square$
\end{prop}

\begin{rem}
Operadic modules as defined for instance in \cite{bergerderived} are also a particular case of this construction. Let $\oper{O}[1]$ be the shift of the operad $\oper{O}$. Explicitely, $\oper{O}[1](n)=\oper{O}(n+1)$ with action induced by the inclusion $\Sigma_n\to\Sigma_{n+1}$. This is in an obvious way a right module over $\oper{O}$. Moreover it has an action of the associative operad
\[\oper{O}[1](n)\times\oper{O}[1](m)\cong\oper{O}(n+1)\times\oper{O}(m+1)\goto{\circ_{n+1}}\oper{O}(n+m+1)=\oper{O}(n+m)[1]\]

It is easy to check that the operad $\oper{O}[1]\oMod$ is the operad parametrizing operadic $\oper{O}$-modules. For instance if $\oper{O}=\oAss$, the associative operad, the category $\oAss[1]\Mod_A$ is the category of $A$-$A$-bimodules. If $\oCom$ is the commutative operad, the category $\oCom[1]\Mod_A$ is the category of left modules over $A$. If $\oper{L}ie$ is the operad parametrizing Lie algebra in an additive symmetric monoidal category, the category $\oper{L}ie[1]\Mod_{\mathfrak{g}}$ is the category of Lie modules over the Lie algebra $\mathfrak{g}$. That is object $M$ equipped with a map
\[-.-:\mathfrak{g}\otimes M\to M\]
satisfying the following relation
\[[X,Y].m=X.(Y.m)-Y.(X.m)\]
\end{rem}

\subsection{Universal enveloping algebra}

We want to show that the category $P\Mod_A$ is the category of left modules over a certain associative algebra built out of $A$ and $P$.

Let $U^{P}_A=P\circ_{\oper{O}}A$. Then by proposition \ref{P structure}, it is an associative algebra in $\cat{C}$

\begin{defi}
The associative algebra $U^P_A$ is called the \emph{universal enveloping algebra of $P\Mod_A$}.
\end{defi}

This name finds its justification in the following proposition.

\begin{prop}\label{universal enveloping algebra}
The category $P\Mod_A$ is equivalent to the category of left modules over the associative algebra $U^P_A$.
\end{prop}

\begin{proof}
We want to apply \cite[Proposition 1.9]{bergerderived} . In the proof of this proposition, we see that if the monad $T$ defining the category $P\Mod_A$ satisfies the condition
\[T(X\otimes Y)\cong T(X)\otimes Y\]
then the category $P\Mod_A$ is equivalent to the category of left modules over $T(\un)$ where the associative algebra structure on $T(\un)$ comes from
\[T(\un)\otimes T(\un)\cong T(T(\un)\otimes \un)\cong T(T(\un))\to T(\un)\]

Let us check that $T$ satisfies $T(X)\otimes Y\cong T(X\otimes Y)$. Note first that if $M$ is in $P\Mod_A$, $M\otimes X$ has an obvious structure of $P$-shaped module over $A$ and that this defines on $P\Mod_A$ the structure of a category tensored over $\cat{C}$. In particular $T(X)\otimes Y$ is in $P\Mod_A$. Let us compute
\[\Map_{P\Mod_A}(T(X)\otimes Y,M)\cong\Map_{\cat{C}}(Y,\Map_{\cat{C}}(X,M))\cong\Map_{C}(X\otimes Y,M)\]
Therefore, the functor represented by $T(X)\otimes Y$ is isomorphic to the functor represented by $T(X\otimes Y)$.

Let us compute $T(\un)$ in our case. 

Let $J$ be the associative algebra in $\Mod_{\oper{O}}$ which sends $0$ to $*$ and everything else to $\varnothing$. $J$ gives rise to a theory of modules. The operad $J\oMod$ has the following description:
\begin{align*}
J\oMod(a^{\boxplus k},a)&=\oper{O}(k)\\
J\oMod(a^{\boxplus k}\boxplus m,m)&=*\;\textrm{if}\;k=\varnothing,\;\varnothing\;\textrm{otherwise}
\end{align*}

The theory of modules parametrized by $J$ is the simplest possible. There are no operations $A^{\otimes n}\otimes M\to M$ except the identity map $M\to M$. 

There is an obvious operad map 
\[J\oMod\to P\oMod\]
inducing a forgetful functor $\cat{C}[P\oMod]\to\cat{C}[J\oMod]$. Let us fix the $\oper{O}$-algebra $A$. One checks easily that $J\Mod_A$ is isomorphic to the category $\cat{C}$. The functor $T$ is the left adjoint of the forgetful functor $P\Mod_A\to J\Mod_A\cong\cat{C}$.

Let us first study the left adjoint $F:\cat{C}[J\oMod]\to\cat{C}[P\oMod]$. This is an operadic left Kan extension. By \ref{coend}, we have the equation
\[F(A,\un)(m)\cong P\Mod(-,m)\otimes_{J\Mod}A^{\otimes-}\]

Note that the only nonempty mapping object in $P\Mod$ with target $m$ are those with source of the form $a^{\boxplus s}\boxplus m$. Hence if we denote $J\Mod_*$ and $P\Mod_*$ the full subcategories with objects of the form $a^{\boxplus s}\boxplus m$, the above coend can be reduced to
\[F(A,\un)(m)\cong P\Mod_*(-,m)\otimes_{J\Mod_*}A^{\otimes-}\]

Let us denote by $\cat{Fin}_*$ the category whose objects are nonnegative integers $n_*$ and whose morphisms from $n_*$ to $m_*$ are morphisms of finite pointed sets
\[\{*,1,\ldots,n\}\to\{*,1,\ldots,m\}\]

The previous coend is the coequalizer
\begin{align*}
\bigsqcup_{f\in\cat{Fin}_*(s_*,t_*)}P(t)\times \left(\prod_{x\in t}\oper{O}(f^{-1}(x))\right)\times J(f^{-1}(*))\otimes A^{\otimes s} \\
 \rightrightarrows \bigsqcup_{s\in\cat{Fin}}P(s)\otimes A^{\otimes s}
\end{align*}

Since the right module $J$ takes value $\varnothing$ for any non-empty set, we see that the  coproduct on the left does not change if we restrict to maps $s_*\to t_*$ for which the inverse image of the base point of $t_*$ is the base point of $s_*$. This set of maps is in bijection with the set of unbased maps $s\to t$. Therefore, the coend can be equivalently written as
\[
\bigsqcup_{f\in\cat{Fin}(s,t)}P(t)\times \left(\prod_{x\in t}\oper{O}(f^{-1}(x))\right)\otimes A^{\otimes s} \rightrightarrows \bigsqcup_{s\in\cat{Fin}}P(s)\otimes A^{\otimes s}
\]
which is the definition of $U^P_A$.

One can compute in a similar but easier fashion that $F(A,\un)(a)\cong A$. 

Hence, we have constructed a natural isomorphism
\[\cat{C}[P\oMod]((A,U^P_A),(A,N))\cong\cat{C}[J\oMod]((A,\un),(A,N))\]

It is clear that this isomorphisms preserves the subset of maps inducing the identity on $A$. Hence we have
\[P\Mod_{A}(U^{P}_A,N)\cong J\Mod_{A}(\un,N)\cong\cat{C}(\un,N)\]
which proves that $T(\un)$ is isomorphic to $U^P_A$.
\end{proof}

\begin{rem}
The above result is well-known if $P=\oper{O}[1]$. See for instance section 4.3. of \cite{fressemodules}.
\end{rem}

\begin{rem}
Note that there is an involution in the category of associative algebras in right modules over $\oper{O}$ sending $P$ to $P\op$. The construction $P\mapsto U^P_A$ sends $P\op$ to $(U^P_A)\op$.
\end{rem}

Assume that $\alpha:\oper{O}\to\oper{Q}$ is a morphism of operads. Let $A$ be an $\oper{Q}$ algebra and $P$ be an associative algebra in right modules over $\oper{O}$. Then by forgetting along the map $\oper{O}\to\oper{Q}$, we construct $\alpha^*A$ which is an $\oper{O}$-algebra and one may talk about the category $P\Mod_{\alpha^*A}$. The following proposition shows that this category of modules is of the form $Q\Mod_A$ for some $Q$.

\begin{prop}\label{change of operad}
We keep the notation of the previous remark. The object $\alpha_!P=P\circ_{\oper{O}}\oper{Q}$ is an associative algebra in right modules over $\oper{Q}$. Moreover, the category $P\Mod_{\alpha^*A}$ is equivalent to the category $\alpha_!P\Mod_A$.
\end{prop}

\begin{proof}
The first part of the claim follows from the fact that $P\circ_{\oper{O}}\oper{Q}$ is a reflexive coequalizer of associative algebras in right $\oper{Q}$-modules and reflexive coequalizers preserve associative algebras.

The second part of the claim follows from a comparison of universal enveloping algebras
\begin{align*}
U^{\alpha_!P}_A&\cong(P\circ_{\oper{O}}\oper{Q})\circ_{\oper{Q}}A\\
 &\cong P\circ_{\oper{O}}(\oper{Q}\circ_\oper{Q} A)\\
 &\cong P\circ_{\oper{O}}\alpha^*A\cong U^{P}_{\alpha^*A}
\end{align*}
\end{proof}

\subsection{Model category structure}

We now give a model structure to the category $P\Mod_A$. In the remaining of this section, $(\cat{C},\otimes,\un_{\cat{C}})$ will denote the symmetric monoidal category $\Mod_E$ where $E$ is a commutative algebra in symmetric spectra. We give it the positive model structure. We restrict our proofs to this case but we have chosen a rather neutral notation to emphasize the fact that the main results work quite generally. In particular, up to minor modifications, our results remain true in  $\cat{Ch_*}(R)$ the category of chain complexes over a commutative $\mathbb{Q}$-algebra $R$. If one is willing to restrict to $\Sigma$-cofibrant operads and modules, they are also true in $\S$ or simplicial $R$-modules for a general commutative ring $R$.

\begin{prop}\label{model category on PMod_A}
Let $\oper{O}$ be an operad and $P$ be a right $\oper{O}$-module. Let $A$ be a cofibrant $\oper{O}$-algebra. There is a model category structure on the category $P\Mod_A$ in which the weak equivalences and fibrations are the weak equivalences and fibrations in $\cat{C}$. 

Moreover, this model structure is simplicial and if $\cat{C}$ is a $\cat{V}$-enriched model category for some monoidal model category $\cat{V}$, then so is $P\Mod_A$.
\end{prop}

\begin{proof}
The category $P\Mod_A$ is isomorphic to $\Mod_{U_A^P}$. Since $\cat{C}$ satisfies the monoid axiom, the result follows \ref{model structure on modules}.
\end{proof}

The category $P\Mod_A$ depends on the variables $P$ and $A$. As expected, there are ``base change'' Quillen adjunctions.

\begin{prop}
Let $(\oper{O},P)\to (\oper{O}',P')$ be a morphism in $\cat{OM}$ and $A$ be a cofibrant $\oper{O}$-algebra, then there is a Quillen adjunction
\[P\Mod_A\leftrightarrows P'\Mod_A\]

Similarly, if $A\to A'$ is a morphisms of cofibrant $\oper{O}$-algebras then there is a Quillen adjunction
\[P\Mod_A\leftrightarrows P\Mod_{A'}\]
\end{prop}

\begin{proof}
In both cases, we get an induced map between the corresponding universal enveloping algebras. The result then follows from \ref{Quillen equivalence module}.
\end{proof}

In some cases these adjunctions are Quillen equivalences.

\begin{prop}\label{pseudo-cofibrantness}
Let $A$ be a cofibrant algebra over $\oper{O}$. The functor $P\otimes_{\cat{O}}A$ sends any right module over $\oper{O}$ to an object of $\cat{C}$ that is absolutely cofibrant.
\end{prop}

\begin{proof}
The operad $\oper{O}_P$ and the functor $\alpha_P$ are as in \cite[Proposition 2.7]{horelfactorization}. If $A$ is cofibrant, then $(\alpha_P)_!(A)$ is cofibrant in $\cat{C}[\oper{O}_P]$. This implies that it is colorwise absolutely cofibrant, in particular, its value at $\infty$  which is just $P\otimes_{\cat{O}}A$ is absolutely cofibrant.
\end{proof}

\begin{prop}
\begin{itemize}
\item If $(\oper{O},P)\to (\oper{O}',P')$ is a weak equivalence in $\cat{OM}$ then there is a Quillen equivalence
\[P\Mod_A\leftrightarrows P'\Mod_A\]
\item If $A\to A'$ is a morphisms of cofibrant $\oper{O}$-algebras then there is a Quillen equivalence
\[P\Mod_A\leftrightarrows P\Mod_{A'}\]
\end{itemize}
\end{prop}

\begin{proof}
This does not quite follow from \ref{Quillen equivalence module} since $U_A^P$ and $U_A^{P'}$ are not cofibrant. However, they are cofibrant in the absolute model structure by \ref{pseudo-cofibrantness}. Therefore, in the first case, using \ref{Quillen equivalence module}, we get a Quillen equivalence
\[P\Mod_A^a\leftrightarrows P'\Mod_A^a\]
where we use the absolute instead of positive model structure. Since moreover, we have Quillen equivalences $P\Mod_A\leftrightarrows P\Mod_A^a$ and $P'\Mod_A\leftrightarrows P'\Mod_A^a$, the result follows. 

The second case is treated in a similar way.
\end{proof}

The following proposition gives a simple description of the cofibrant objects of the model category $\cat{C}[P\oMod]$ whose algebra component is cofibrant.

\begin{prop}\label{cofibrant replacement in oMod}
Let $A$ be a cofibrant $\oper{O}$-algebra in $\cat{C}$. Let $M$ be an object of $P\Mod_A$. The pair $(A,M)$ is a cofibrant object of $\cat{C}[P\oMod]$ if and only if $M$ is a cofibrant object of $P\Mod_A$.
\end{prop}

\begin{proof}
Assume $(A,M)$ is cofibrant in $\cat{C}[P\oMod]$. For any trivial fibration $N\to N'$ in $P\Mod_A$, the map $(A,N)\to (A,N')$ is a trivial fibration in $\cat{C}[P\oMod]$. A map of $P$-shaped $A$-module $M\to N'$ induces a map of $P\oMod$-algebras $(A,M)\to (A,N')$ which can be lifted to a map $(A,M)\to (A,N)$ and this lift has to be the identity on the first component. Thus $M$ is cofibrant.

Conversely, let $(B',N')\to (B,N)$ be a trivial fibration in $\cat{C}[P\oMod]$. We want to show that any map $(A,M)\to (B,N)$ can be lifted to $(B',N')$. We do this in two steps. We first lift the first component and then the second component.

Note that if we have a map $A\to B$, any $P$-shaped module $N$ over $B$ can be seen as a $P$-shaped module over $A$ by restricting the action along this map. With this in mind, it is clear that any map $(A,M)\to (B,N)$ can be factored as
\[(A,M)\to(A,N)\to (B,N)\]
where the first map is a map in $P\Mod_A$ and the second map induces the identity on $N$.

Since the map $(B',N')\to (B,N)$ is a trivial fibration in $\cat{C}[P\oMod]$, the induced map $B'\to B$ is a trivial fibration in $\cat{C}$ which implies that it is a trivial fibration in $\cat{C}[\oper{O}]$. $A$ is cofibrant as an $\oper{O}$-algebra so we can find a factorization $A\to B'\to B$.

Using this map, we can see $N'$ as an object of $P\Mod_A$ and, we have the following diagram in $\cat{C}[P\oMod]$:

\[\xymatrix{
 &(A,N')\ar[d]\ar[r]& (B',N')\ar[d]\\
(A,M)\ar[r]&(A,N)\ar[r]&(B,N)}\]

We want to construct a map $(A,M)\to (A,N')$ making the diagram to commute. The map $(A,N')\to (A,N)$ is the product of the identity of $A$ and a trivial fibration $N\to N'$ in $\cat{C}$. This implies that $(A,N')\to (A,N)$ is a trivial fibration in $P\Mod_A$, hence we can construct a map $(A,M)\to (A,N')$ making the left triangle to commute, which gives us the desired lift $(A,M)\to (B',N')$.
\end{proof}

\subsection{Operations on modules}

Let $\oper{O}$ be a single-object operad in $\S$. We denote by $\Mod_{\oper{O}}$ the category of right $\oper{O}$-modules with the \emph{injective} model structure. That is the model structure in which the weak equivalences and cofibrations are those maps that are objectwise weak equivalences or cofibrations.

\begin{prop}
The category $\Mod_{\oper{O}}$, is a symmetric monoidal model category.
\end{prop}

\begin{proof}
All the functor categories in this proof are equipped with the injective model structure.

(1) Since the tensor product and colimits of right $\oper{O}$-modules are computed in the category of symmetric sequences, and since the forgetful functor $\Mod_{\oper{O}}\to\Mod_{\oper{I}}$ reflects cofibrations and weak equivalences, we can assume that $\oper{O}$ is the terminal operad and we just need to prove that $\on{Fun}(\Sigma\op,\S)$ with the injective model structure is symmetric monoidal. 

(2) Let $\on{Ind}$ be the functor from $\Sigma_n\times\Sigma_m$-spaces to $\Sigma_{n+m}$-spaces which is left adjoint to the the forgetful functor. Concretely, $\on{Ind}(X)=(X\times\Sigma_{n+m})/(\Sigma_n\times\Sigma_m)$. Notice that $\on{Ind}$ sends monomorphisms to monomorphisms and preserves weak equivalences. Moreover it is easy to check that the pairing
\[-\otimes_{n,m}-:\S^{\Sigma_n}\times\S^{\Sigma_m}\to\S^{\Sigma_{n+m}}\]
sending $(X,Y)$ to $\on{Ind}(X\times Y)$ is a left Quillen bifunctor.

(3) For a $\Sigma_n$-space $X$, we denote by $\sigma_n(X)$ the symmetric sequence which is $X$ in degree $n$ and $\varnothing$ in all other degree. Any cofibration in $\on{Fun}(\Sigma\op,\S)$ is a coproduct of maps of the form $\sigma_n(X)\to \sigma_n(Y)$ for $X\to Y$ an injection of $\Sigma_n$-spaces. Hence, it suffices to prove the pushout-product axiom in $\on{Fun}(\Sigma\op,\S)$ for those maps. 

(4) Let $X$ be a $\Sigma_n$-space and $Y$ be a $\Sigma_m$-space. Then we have
\[\sigma_n(X)\otimes\sigma_m(Y)\cong\sigma_{n+m}(X\otimes_{n,m}Y)\]

Let $f:X\to Y$ be an injection of $\Sigma_n$-spaces and $g:Z\to T$ be an injection of $\Sigma_m$-spaces. Then the pushout product map is
\[\sigma_{n+m}(X\otimes_{n,m}T\sqcup^{X\otimes_{n,m}Z} Y\otimes_{n,m}Z)\to\sigma_{n+m}(Y\otimes_{n,m}T)\]
which is a cofibration because of part (2). Moreover if $f$ or $g$ is a weak equivalence, so is the pushout product also because of part (2).
\end{proof}

In this symmetric monoidal model category, any algebra is admissible and hence we can talk about the bioperad $\oMor(\Mod_{\oper{O}})$ (see \ref{Morita bioperad} for the construction) whose class of objects is the class of all associative algebras in $\Mod_{\oper{O}}$.

\begin{prop}\label{monoidality}
Let $A$ be a cofibrant $\oper{O}$-algebra, then the functor 
\[-\otimes_{\cat{O}}A:\Mod_{\oper{O}}\to\cat{C}\]
is a symmetric monoidal functor which preserves all weak equivalences between objects of $\Mod_{\oper{O}}$.
\end{prop}

\begin{proof}
The fact about equivalences is proved in \ref{invariance of operadic coend}. Let us prove that the functor is symmetric monoidal. Let $P$ and $Q$ be two right $\oper{O}$-modules. Then $(P\otimes Q)\circ_{\oper{O}}A\cong (P\otimes Q)\otimes_{\cat{O}}A$ is the coequalizer of
\[(P\otimes Q)\circ\oper{O}\circ A\rightrightarrows (P\otimes Q)\circ A\]
But for any symmetric sequence $S$, we have the formula
\[(P\otimes Q)\circ S\cong (P\circ S)\otimes(Q\circ S)\]
proved for instance in \cite{fressemodules}. Hence the coequalizer is isomorphic to the coequalizer of
\[[(P\circ\oper{O})\otimes(Q\circ\oper{O})]\circ A\rightrightarrows (P\otimes Q)\circ A\]
which, applying again the formula is isomorphic to the coequalizer of
\[(P\circ\oper{O}\circ A)\otimes(Q\circ\oper{O}\circ A)\rightrightarrows (P\circ A)\otimes (Q\circ A)\]
which is precisely $(P\circ_{\oper{O}}A)\otimes (Q\circ_\oper{O}A)$.
\end{proof}

Since in $\cat{C}$, all algebras are admissible, we have by \ref{Morita functoriality}, a morphism of bioperads
\[\oMor(\Mod_{\oper{O}})\to\oMor(\cat{C})\]

The main result of this section is the following:

\begin{theo}\label{map from Mor to ModCat}
The assignment $P\mapsto P\Mod_A$ defines a $\oMor(\oper{O})$ (resp. $\oMor(\oper{O},\Sigma)$)-algebra in $\oModCat$.
\end{theo}

\begin{proof}
According to \ref{strictification}, it suffices to construct a map of bioperad
\[\oMor(\oper{O})\to\oModCat\]
But it suffices to define it as the composition of the map
\[\oMor(\oper{O})\to\oMor(\cat{C})\]
constructed above, followed by the map
\[\oMor(\cat{C})\to\oModCat\]
sending $R$ to $L\Mod_R$ constructed in \ref{from rings to cat}.
\end{proof}

\subsection{Homotopy invariance}

It remains to check how ``model independant'' this structure is.

First we observe that if $A\to A'$ is a map of cofibrant $\oper{O}$-algebras, we get a natural transformation of functors
\[-\circ_{\oper{O}}A\to -\circ_\oper{O}A'\]
If the map was a weak equivalence, this natural transformation is a weak equivalence. In particular, we find that the functor induced by $A$ and $A'$:
\[\oMor(\oper{O})\to\oMor(\cat{C})\]
are equivalent up to a natural $2$-morphism.

Now we want to talk about change of operads.

\begin{prop}
Let $u:\oper{P}\to\oper{Q}$ be a map of operad. Then, there are two maps of bioperads
\[u_!:\oMor(\oper{P})\to\oMor(\oper{Q})\]
and
\[u^*:\oMor(\oper{Q})\to\oMor(\oper{P})\]
that are inverse weak equivalences if $u$ is a weak equivalence.
\end{prop}

\begin{proof}
This follows from \ref{Morita functoriality} and \ref{Morita equivalence}.
\end{proof}

In particular, if $\oper{P}\to\oper{Q}$ is a weak equivalence and $A$ is a cofibrant $\oper{P}$-algebra, then there is a diagram
\[\xymatrix{
\oMor(\oper{P})\ar[r]^{u_!}\ar[dr]_A&\oMor(\oper{Q})\ar[d]^{u_!A}\\
                            &\oMor(\cat{C})
}
\]
which commutes up to a natural weak equivalence. Indeed, for any right module $M$ over $\oper{P}$, we have an equivalence
\[M\circ_{\oper{P}} A\to u^*u_!M\circ_{\oper{P}} A\cong u_!M\circ_{\oper{Q}}u_!A\]

\section{Factorization homology for manifolds with singularities}

In this section, we define factorization homology for manifolds with a fixed boundary. A much more general treatement of factorization homology for singular manifolds can be found in \cite{ayalastructured}. The only originality of the present section is the use of model category techniques as opposed to $\infty$-categories.

In this section, $\epsilon$ denotes a fixed real number in the open interval $(0,1)$.

Note that since all operads and modules appearing in this section are $\Sigma$-cofibrant, we could replace the category $cat{C}$ by the category $\S$ or the category of simplicial $R$-modules over any commutative ring.

\subsection{Embeddings between structured manifolds}

This subsection owes a lot to \cite{andrademanifolds}. In particular, the definition \ref{framed embedding} can be found in that reference. We then make analogous definitions of embedding spaces for $S_{\tau}$-manifolds which are straightforward generalizations of Andrade's construction.

\begin{defi}
A \emph{framed $d$-manifold} is a pair $(M,\sigma_M)$ where $M$ is a $d$-manifold and $\sigma_M$ is a smooth section of the $\on{GL}(d)$-principal bundle $\on{Fr}(TM)$.
\end{defi}

If $M$ and $N$ are two framed $d$-manifolds, we define a space of framed embeddings denoted by $\Emb_f(M,N)$ as in \cite{andrademanifolds}: 

\begin{defi}\label{framed embedding}
Let $M$ and $N$ be two framed $d$-dimensional manifolds. The \emph{topological space of framed embeddings from $M$ to $N$}, denoted $\Emb_f(M,N)$, is given by the following homotopy pullback in the category of topological spaces over $\Map(M,N)$:
\[\xymatrix
{\Emb_f(M,N)\ar[r]\ar[d] & \Map(M,N)\ar[d]\\
\Emb(M,N)\ar[r] & \Map_{\on{GL}(d)}(\on{Fr}(TM),\on{Fr}(TN))}
\]

The right hand side map is obtained as the composition
\[\Map(M,N)\to\Map_{\on{GL}(d)}(M\times\on{GL}(d),N\times\on{GL}(d))\cong\Map_{\on{GL}(d)}(\on{Fr}(TM),\on{Fr}(TN))\]  
where the first map is obtained by taking the product with $\on{GL}(d)$ and the second map is induced by the identification $\on{Fr}(TM)\cong M\times\on{GL}(d)$ and $\on{Fr}(TN)\cong N\times \on{GL}(d)$.
\end{defi}

A priori, this only defines $\Emb_f(M,N)$ as a homotopy type however, we can choose a certain explicit model. This explicit model allows us to construct well defined composition maps
\[\Emb_f(M,N)\times\Emb_f(N,P)\to\Emb_f(M,P)\]
allowing the construction of a topological category $f\Man_d$ (see \cite{andrademanifolds} for a precise construction of this category).

We now want to define a category of manifolds whose objects are manifodls with a fixed boundary.

\begin{defi}
An \emph{$S$-manifold} is a pair $(M,f)$ where $M$ is a $d$-manifold with boundary and $f:S\times[0,\epsilon)\to M$ is an embedding whose restriction to the boundary induces a diffeomorphism $S\cong \partial M$ 
\end{defi}

If $M$ and $N$ are two $S$-manifolds, we denote by $\Emb^S(M,N)$ the topological space of embeddings which commute strictly with the structure maps $S\times[0,\epsilon)\to M$ and $S\times[0,\epsilon)\to N$. We give it the weak $C^1$-topology.

\begin{defi}
A \emph{$d$-framing} of a $(d-1)$-manifold $S$ is a trivialization of the  $d$-dimensional bundle $TS\oplus\mathbb{R}$ where $\mathbb{R}$ is a trivial line bundle. 
\end{defi}

Note that if $\tau$ is a $d$-framing of $S$, then $\tau$ canonically induces a framing of $S\times[0,\epsilon)$.

\begin{defi}
A \emph{framed $S_\tau$-manifold} is an $S$-manifold $(M,f)$ with the datum of a framing of $TM$ such that the embedding
\[f:S\times [0,\epsilon)\to M\]
preserves the framing on the nose when we give $S\times[0,\epsilon)$ the framing $\tau$.
\end{defi}

\begin{defi}
Let $(M,f)$ and $(N,g)$ be two framed $S_\tau$-manifolds. The \emph{topological space of framed embeddings from $M$ to $N$}, denoted $\Emb_f^{S_\tau}(M,N)$, is the following homotopy pullback taken in the category of topological spaces over $\Map^S(M,N)$:
\[\xymatrix
{\Emb_f^{S_\tau}(M,N)\ar[r]\ar[d] & \Map^S(M,N)\ar[d]\\
\Emb^S(M,N)\ar[r] & \Map^{S_\tau}_{\on{GL}(d)}(\on{Fr}(TM),\on{Fr}(TN))}
\]

Any time we use the $S$ superscript, we mean that we are considering the subspace of maps commuting with the given map from $S\times[0,\epsilon)$. The topological space in the lower right corner is the space of morphisms of $\on{GL}(d)$-bundles inducing the identity over $S\times[0,\epsilon)$.
\end{defi}

We can extend the notation $\Emb^S(-,-)$ to manifolds without boundary:
\begin{itemize}
\item $\Emb^S(M,N)=\Emb(M,N)$ if $M$ is a manifold without boundary and $N$ is either an $S$-manifold or a manifold without boundary.
\item $\varnothing$ if $M$ is an $S$-manifold and $N$ is a manifold without boundary.
\end{itemize}

Using these as spaces of morphisms, there is a simplicical category $\Man^S_d$ whose objects are $S$-manifolds. Similarly, we can extend the notation $\Emb^{S_\tau}_f(-,-)$  to framed manifolds without boundary as above and construct a simplicical category $f\Man^{S_\tau}_d$ whose objects are framed $S_\tau$-manifolds.

\begin{rem}
When there is no ambiguity, we sometimes allow ourselves to drop the framing notation and write $S$ instead of $S_\tau$ to keep the notation simple.
\end{rem}

\subsection{Definition of the right module associated to a $(d-1)$-manifold}

In this subsection, we show that each $d$-framed $(d-1)$-manifold gives rise to a right module over the operad $\oper{E}_d$. Let us first recall the definition of $\oper{E}_d$ which will be used in this paper.

\begin{defi}
The operad $\oper{E}_d$ of \emph{little $d$-disks} is the simplicial operad whose $n$-th space is (the singular simplicial set on) $\Emb_f(D^{\sqcup n},D)$.
\end{defi} 

According to \ref{framed embeddings of disks}, the $k$-th space of this operad is equivalent to the space of configurations of $k$ points in the disk. Using this it is not hard to check that this model of $\oper{E}_d$ is equivalent to any other definition of the little $d$-disks operad.

\medskip

Let $S$ be a $(d-1)$-manifold and let $\tau$ be a $d$-framing of $S$.

\begin{defi}\label{def of Stau modules}
The right $\oper{E}_d$-module $\und{S}_\tau$ is given by
\[\und{S}_\tau(n)=\Emb_f^{S_\tau}(D^{\sqcup n}\sqcup S\times[0,1),S\times[0,1))\]
\end{defi}

We now construct an algebra structure on this right $\oper{E}_d$-module.

\begin{nota}
It will be convenient to have the following notation at our disposal. Assume that $M$, $N$, $P$ and $Q$ are manifolds. Let $\phi:M\to N$ and $\psi:N\sqcup P\to Q$ be smooth maps, the map $\psi\circ\phi$ is by convention the map $M\sqcup P\to Q$ given as the following composite
\[M\sqcup P\goto{\phi\sqcup P}N\sqcup P\goto{\psi}Q\]
Notice that if both $\phi$ and $\psi$ were embeddings, the resulting map $\psi\circ\phi$ is still an embedding. Moreover, if $\phi$ and $\psi$ live in one of the space of ``framed embeddings'' defined above the composite $\psi\circ \phi$ can be seen to live in the appropriate space.
\end{nota}

\begin{cons}
We construct a map $\und{S}_\tau\otimes\und{S}_\tau\to\und{S}_\tau$. By definition of the tensor product, it suffices to construct a graded monoid structure on $\{\und{S}_\tau(n)\}_{n\geq 0}$.
If $\phi$ is a point in $\und{S}_\tau(n)$ and $\psi $ is a point in $\und{S}_\tau(m)$, then $\psi\circ\phi$ is a point in $\und{S}_\tau(m+n)$ and this defines a pairing
\[\und{S}_\tau(n)\times\und{S}_\tau(m)\to\und{S}_\tau(n+m)\]
which is easily seen to extend to an associative algebra structure on $\und{S}_\tau$.
\end{cons}

The general theory of the first section gives rise to an operad $\und{S}_\tau \oMod$ and for any $\oper{E}_d$-algebra $A$ in $\cat{C}$, a category $\und{S}_\tau\Mod_A$.

\medskip

The unit sphere inclusion $S^{d-1}\to\mathbb{R}^d$ has a trivial normal bundle. This induces a $d$-framing on $S^{d-1}$ which we denote $\kappa$. Using \ref{def of Stau modules}, we can construct an operad $\und{S}^{d-1}_\kappa\oMod$. We have the following result relating $\und{S}^{d-1}_\kappa$-modules to operadic modules over $\oper{E}_d$.

\begin{prop}\label{equivalence of modules}
$\und{S}^{d-1}_\kappa$ and $\oper{E}_d[1]$ are weakly equivalent as associative algebras in right modules over $\oper{E}_d$. In particular, for a cofibrant $\oper{E}_d$-algebra $A$, the category $S_\kappa^{d-1}\Mod_A$ is connected to $\oper{E}_d[1]\Mod_A$ by a Quillen equivalence.
\end{prop}

\begin{proof}
We construct a chain of weak equivalences:
\[\und{S}^{d-1}_\kappa\rightarrow\oper{E}_d^*\rightarrow \oper{E}_d[1]\]

Let $D=\mathbb{R}^d$ and $D^*$ be the manifold $D$ based at the origin. The space $\oper{E}_d^{*}(n)$ is the space of framed embeddings $D_*\sqcup D^{\sqcup n}\to D$ preserving the base point. Any embedding $S^{d-1}\times[0,1)\to S^{d-1}\times[0,1)$ can be extended to an embedding $D^*\to D^*$ by glueing a closed disk along the boundary. This construction easily extends to give a map of associative algebras in right modules over $\oper{E}_d$:
\[\und{S}^{d-1}_\kappa\to\oper{E}_d^*\]
We claim that this map is an equivalence. It suffices to check it in each degree. There is a map \[\oper{E}_d^*(n)\to \Emb(D^{\sqcup n},D-0)\]
which sends an embedding to its restriction on the $n$ non-pointed disks. This map is a fibration whose fiber  over $\phi$ is the space of based embedding of the pointed disk in the complement of the configuration of disk given by the embedding $\phi$. It is straightforward to check that this space is contractible. Similarly, the composite
\[\und{S}^{d-1}_\kappa(n)\to\oper{E}_d^*(n)\to \Emb(D^{\sqcup n},D-0)\]
is a fibration with contractible fibers (by \ref{contractibility collars framed}).

A point in $\oper{E}_d^*(n)$ is just a point in $\oper{E}_d[1]$ where the orgin of the special disk is sent to the origin. Hence we have an inclusion
\[\oper{E}_d^*\to\oper{E}_d[1]\]
which also preserves the structure of an associative algebra in right modules. To prove that this map is an equivalence, it suffices to do it in each degree. But for an integer $n$, the map
\[\oper{E}_d^*(n)\to \oper{E}_d[1](n)\]
can be obtained as the pullback
\[
\xymatrix{
\oper{E}_d^*(n)\ar[d]\ar[r]&\oper{E}_d[1](n)\ar[d]^p\\
\on{pt}\ar[r] & D
}
\]
where the map $p$ evaluates an embedding at the center of the marked disk. It is classical that the map $p$ is a fibration. Hence by right properness of the category of spaces, the map $\oper{E}_d^*(n)\to\oper{E}_d[1](n)$ is an equivalence.
\end{proof}

\subsection{Factorization homology}

Let $\mathsf{M}$ be the set of framed $d$ manifolds whose underlying manifold is a submanifold of $\mathbb{R}^{\infty}$. Note that $\mathsf{M}$ contains at least one element of each diffeomorphism class of framed $d$-manifold.

\begin{defi}
We denote by $f\oMan_d$ \emph{the operad of framed $d$-manifolds} whose set of objects is $\mathsf{M}$ and with mapping space
\[f\oMan_d(\{M_1,\ldots,M_n\},M)=\Emb_f(M_1\sqcup\ldots\sqcup M_n,M)\]

As usual, we denote by $f\Man_d$ the free symmetric monoidal category on the operad $f\oMan_d$.

We can see the open unit disk $D\subset\mathbb{R}^d\subset\mathbb{R}^{\infty}$ as an element of $\mathsf{M}$. We denote by $\oper{E}_d$ the full suboperad of $f\oMan_d$ on the object $D$.  The category $\cat{E}_d$ is the full subcategory of $f\Man_d$ on objects of the form $D^{\sqcup n}$ with $n$ a nonnegative integer.
\end{defi}

Let $S$ be a $(d-1)$-manifold and $\tau$ be a $d$-framing on $S$. Let $\mathsf{M}^{S_\tau}$ be the set of $S_\tau$-manifolds whose underlying manifold is a submanifold of $\mathbb{R}^\infty$.

\begin{defi}
The \emph{operad of framed $S_\tau$-manifolds} denoted $f\oMan_d^{S_\tau}$ has the set $\mathsf{M}\sqcup\mathsf{M}^{S_\tau}$ as set of objects. Its spaces of operations are given by:
\begin{align*}
f\oMan_d^{S_\tau}(\{M_i\}_{i\in I};N)&=\varnothing \textrm{, if}\;\{M_i\}_{i\in I}\textrm{ contains more than 1 element of}\;\mathsf{M}^{S_\tau}\\
 &=\Emb_f^{S_\tau}(\sqcup_i M_i,N)\textrm{ otherwise}
\end{align*}
\end{defi}

We see that the operad $\und{S}_\tau\oMod$ (defined in \ref{def of Stau modules}) sits inside $f\oMan_d^{S_{\tau}}$ as the full suboperad on the objects $D$ and $S\times[0,1)$. This observation can be used to define factorization homology over $S_\tau$-manifolds.

\begin{defi}\label{factorization}
Let $A$ be an object of $\cat{C}[\oper{E}_d]$. We define the \emph{factorization homology with coefficients in $A$} to be the derived operadic left Kan extension of $A$ along the map of operads $\oper{E}_d\to f\oMan_d$. 

We denote by $M\mapsto\int_MA$ the symmetric monoidal functor $f\Man_d\to\cat{C}$ induced by that pushforward.
\end{defi}

We have $\int_MA=\Emb_f(-,M)\otimes_{\cat{E}_d}QA$ where $QA\to A$ is a cofibrant replacement in the category $\cat{C}[\oper{E}_d]$. We use the fact that the operad $\oper{E}_d$ is $\Sigma$-cofibrant and that the right module $\Emb_f(-,M)$ is $\Sigma$-cofibrant.

We can define, in a similar fashion, factorization homology on an $S_\tau$-manifold. This gives a pairing between $S_\tau$-manifolds and $\und{S}_{\tau}\oMod$-algebras.

\begin{defi}
Let $(A,M)$ be an $\und{S}_\tau\oMod$-algebra in $\cat{C}$. \emph{Factorization homology with coefficients in $(A,M)$} is the derived operadic left Kan extension of $(A,M)$ along the map of operad
\[\und{S}_\tau\oMod\to f\oMan_d^{S_\tau}\]

We write $\int_W(A,M)$ for the value at $W\in \und{S}_\tau\oMod$ of factorization homology with coefficients in $(A,M)$.
\end{defi}

\subsection{Factorization homology as a homotopy colimit}

In this subsection, we show that factorization homology can be expressed as the homotopy colimit of a certain functor on the poset of open sets of $M$ that are diffeomorphic to a disjoint union of disks. Note that this result in the case of manifolds without boundary is proved in \cite{luriehigher}.

Let $M$ be an object of $ f\Man_d$. Let $\cat{D}(M)$ the poset of subset of $M$ that are diffeomorphic to a disjoint union of disks. Let us choose for each object $V$ of $\cat{D}(M)$ a framed diffeomorphism $V\cong D^{\sqcup n}$ for some uniquely determined $n$. Each inclusion $V\subset V'$ in $\cat{D}(M)$ induces a morphism $D^{\sqcup n}\to D^{\sqcup n'}$ in $\cat{E}_d$ by composing with the chosen parametrization. Therefore each choice of parametrization induces a functor $\cat{D}(M)\to\cat{E}_d$. Up to homotopy this choice is unique since the space of automorphisms of $D$ in $ \cat{E}_d$ is contractible. 

In the following we assume that we have one of these functors $\delta:\cat{D}(M)\to\cat{E}_d$. We fix a cofibrant algebra $A: \cat{E}_d\to\cat{C}$.

\begin{prop}\label{nb}
There is a weak equivalence:
\[\int_MA\simeq\on{hocolim}_{V\in\cat{D}(M)} A(\delta (V))\]
\end{prop}

\begin{proof}
See \cite[Corollary 7.7]{horelfactorization}
\end{proof}

We have an analogous assertion for $S_\tau$-manifolds. Let $S_\tau$ be a $d$-framed $(d-1)$-manifold. Let $W$ be an $S_\tau$-manifold. Let $\cat{D}(W)$ be the poset of open subsets of $W$ which are diffeomorphic to $S\times[0,1)\sqcup D^{\sqcup n}$ under a diffeomorphism of $S_\tau$-manifold. Let $\delta:\cat{D}(W)\to \und{S}_\tau\oMod$ be any parametrization. As before, it turns out that the space of choices of such parametrizations is contractible.

\begin{prop}\label{nb2}
 Let $(A,M)$ be a cofibrant algebra over $\und{S}_\tau\oMod$. Then there is a weak equivalence:
\[\int_W(A,M)\simeq \on{hocolim}_{U\in\cat{D}(W)}(A,M)(\delta(V))\]
\end{prop}

\begin{proof}
A very similar statement is proved in \cite{horelfactorization}. We start by proving that the right $\und{S}_\tau\oMod$-module $\Emb_f^{S_\tau}(-,W)$ is the homotopy colimit over $U\in\cat{D}(W)$ of $\Emb_f^{S_\tau}(-,U)$ and this is easily done analogously to \cite[lemma 7.9.]{horelfactorization}. Then one proceeds as in \cite[proposition 7.10]{horelfactorization}.
\end{proof} 

There is an obvious map $S_\tau\to\Emb_f(-S\times(0,1))$ which forgets the part with boundary and restricts to the open disks.

\begin{prop}
The map $\und{S}_\tau\to \Emb_f(-,S\times(0,1))$ is a weak equivalence of right $\oper{E}_d$-modules
\end{prop}

\begin{proof}
This is clear.
\end{proof}

\begin{coro}
For a cofibrant $\oper{E}_d$-algebra $A$, there is a weak equivalence
\[U_A^{S_\tau}\goto{\simeq}\int_{S\times(0,1)}A\]
\end{coro}

\begin{proof}
By the previous proposition, there is a weak equivalence of right $\oper{E}_d$-modules
\[\und{S}_\tau\goto{\simeq} \Emb_f(-,S\times(0,1))\]
Then it suffices to apply \ref{invariance of operadic coend} to this map.
\end{proof}

\section{Factorization homology of commutative algebras}

This section is an interlude in our study of the categories of modules over $\oper{E}_d$-algebras. We focus on the simpler case of commutative algebras. It turns out that commutative algebras have a notion of factorization homology where the manifolds can be replaced by simpliclial sets. The definition is a straightforward variant of factorization homology. Such a construction was made by Pirashvili (see \cite{pirashvilihodge}) in the category of chain complexes over a field of characteristic zero. See also \cite{ginothigher}.

Note that since the commutative operad is not $\Sigma$-cofibrant the results presented in this section only work in model categories like symmetric spectra or chain complexes in characteristic $0$ where commutative algebra are homotopically well-behaved. One would need to work with an $\oper{E}_\infty$-operad to make them work in model categories like $\S$ or simplicial modules over a ring.

\subsection{Construction}

Let $\mathsf{S}$ be a set of connected simplicial sets containing the point, we denote $\S^{\mathsf{S}}$ the  operad with objects $\mathsf{S}$ and with spaces of operations:
\[\oSpace^{\mathsf{S}}(\{s_i\}_{i\in I};t):=\Map(\sqcup_I s_i,t)\]

Note that the full suboperad on the point is precisely the operad $\oCom$, therefore, we have a morphism of  operads:
\[\oCom\to\oSpace^{\mathsf{S}}\]

\begin{defi}
Let $A$ be a commutative algebra in $\cat{C}$, let $X$ be an object of the symmetric monoidal category $\cat{Space}^{\mathsf{S}}$, we define the \emph{factorization homology of $A$ over $X$} to be the value at $X$ of the operadic left Kan extension of $A$ along the map 
\[\oCom\to\oSpace^{\mathsf{S}}\]
We denote by $\int_XA$ this object of $\cat{C}$.
\end{defi}

Note that the value of $\int_XA$ is:
\[\Map(-,X)\otimes_{\cat{Fin}}QA\]
where $QA\to A$ is a cofibrant replacement of $A$ as a commutative algebra. In particular, it is independant of the set $\mathsf{S}$. In the following we will write $\int_X A$ for any simplicial set $X$ without mentioning the set $\mathsf{S}$.

\begin{prop}\label{commutative fh preserves weak equivalences}
The functor $X\mapsto\int_X A$ preserves weak equivalences.
\end{prop}

\begin{proof}
The functor $X\mapsto \Map(-,X)$ sends any weak equivalence in $\S$ to a weak equivalence in $\on{Fun}(\cat{Fin}\op,\S)$. The result then follows from \ref{invariance of operadic coend}.
\end{proof}

We now want to compare $\int_XA$ with $\int_MA$ where $M$ is a framed manifold.

Recall that $\cat{D}(M)$ denotes the poset of open sets of $M$ that are diffeomorphic to a disjoint union of disks.

\begin{prop}
There is a weak equivalence
\[\on{hocolim}_{\cat{D}(M)}\cat{Fin}(S,\pi_0(-))\simeq \Map(S,M)\]
\end{prop}

\begin{proof}
Note that for $U\in\cat{D}(M)$, we have $\cat{Fin}(S,\pi_0(U))\simeq\Map(S,U)$, thus, we are reduced to showing:
\[\on{hocolim}_{U\in\cat{D}(M)}\Map(S,U)\simeq \Map(S,M)\]
We use \cite[Theorem A.3.1. p. 971]{luriehigher}, there is a functor $\cat{D}(M)\to\cat{U}(\Map(S,M))$ sending $U$ to the open set of maps whose image is contained in $U$. For $f\in \Map(S,M)$, the subcategory of $U\in\cat{D}(M)$ containing the image of $f$ is filtered, therefore, it is contractible.
\end{proof}

Let $F$ be any functor $\cat{Fin}\to\cat{C}$. We have the following diagram:
\[\cat{D}(M)\stackrel{\alpha}{\rightarrow}\cat{Fin}\stackrel{F}{\rightarrow}\cat{C}\]

\begin{prop}
There is a weak equivalence:
\[\on{hocolim}_{\cat{D}(M)}\alpha^*F\simeq \Map(-,M)\otimes_{\cat{Fin}}^{\L}F\]
\end{prop}

\begin{proof}
The hocolim can be written as a coend
\[*\otimes_{\cat{D}(M)}^{\L}\alpha^*F\]

We use the adjuction induced by $\alpha$, and find:
\[\on{hocolim}_{\cat{D}(M)}\alpha^*F\simeq \L\alpha_!(*)\otimes_{\cat{Fin}}F\]

But $\L\alpha_!(*)$ is the functor whose value at $S$ is:
\[\cat{Fin}\op(\pi_0(-),S)\otimes^{\L}_{\cat{D}(M)\op}*\simeq\on{hocolim}_{\cat{D}(M)}\cat{Fin}(S,\pi_0(-))\]
The results then follows from the previous lemma.
\end{proof}

\begin{coro}
Let $M$ be a framed manifold and $A$ a commutative algebra in $\cat{C}$, then $\int_{\on{Sing}(M)}A$ is weakly equivalent to $\int_MA$
\end{coro}

\begin{proof}
We have by \ref{nb}
\[\int_M A\simeq \on{hocolim}_{\cat{D}(M)}\alpha^*A\]

By \ref{Operadic vs categorical}:
\[\int_{\on{Sing}(M)}A\simeq \Map(-,\on{Sing}(M))\otimes^{\L}_{\cat{Fin}} A\]

Hence the result is a trivial corollary of the previous proposition.
\end{proof}

\subsection{Comparison with McClure, Schw\"anzl and Vogt description of $THH$. }

In \cite{mcclurethh}, the authors show that $THH$ of a commutative ring spectrum $R$ coincides with the tensor $S^1\otimes R$ in the simplicial category of commutative ring spectra. We want to generalize this result and show that for a commutative algebra $A$, there is a natural weak equivalence of commutative algebras:
\[\int_{X}A\simeq X\otimes A\]

Let $X$ be a simplicial set. There is a category $\Delta/X$ called the category of simplices of $X$ whose objects are pairs $([n],x)$ where $x$ is a point of $X_n$ and whose morphisms from $([n],x)$ to $([m],y)$ are maps $d:[n]\to[m]$ in $\Delta$ such that $d^*y=x$. Note that there is a functor:
\[F_X:\Delta/X\to \S\]
sending $([n],x)$ to $\Delta[n]$. The colimit of that functor is obviously $X$ again.

\begin{prop}
The map
\[\on{hocolim}_{\Delta/X}F_X\to\on{colim}_{\Delta/X}F_X\cong X\]
is a weak equivalence.
\end{prop} 

\begin{proof}
see \cite[Proposition 4.2.3.14.]{lurietopos}.
\end{proof}

\begin{coro}\label{unicity of functors}
Let $U$ be a functor from $\S$ to a model category $\cat{Y}$. Assume that $U$ preserves weak equivalences and homotopy colimits. Then $U$ is weakly equivalent to:
\[X\mapsto \on{hocolim}_{\Delta/X}U(*)\]
In particular, if $U$ and $V$ are two such functors, and $U(*)\simeq V(*)$, then $U(X)\simeq V(X)$ for any simplicial set $X$.
\end{coro}

\begin{proof}
Since $U$ preserves weak equivalences and homotopy colimits, we have a weak equivalence:
\[\on{hocolim}_{\Delta/X}U(*)\simeq U(\on{hocolim}_{\Delta/X}*)\simeq U(X)\]
\end{proof}

We now have the following theorem:

\begin{theo}
Let $A$ be a cofibrant commutative algebra in $\cat{C}$. The functor $X\mapsto \int_XA$ and the functor $X\mapsto X\otimes A$ are weakly equivalent as functors from $\S$ to $\cat{C}[\oCom]$.
\end{theo}

\begin{proof}
The two functors obviously coincide on the point. In order to apply \ref{unicity of functors}, we need to check that both functors preserve weak equivalences and homotopy colimits. 

Since $A$ is cofibrant and $\cat{C}$ is simplicial, $X\mapsto X\otimes A$ is a left Quillen functor $\S\to\cat{C}[\oCom]$ and as such preserves weak equivalences and homotopy colimits.

The functor $X\mapsto \int_X A$ preserves weak equivalences by \ref{commutative fh preserves weak equivalences}. We need to show that it preserves homotopy colimits.

Note that if $\cat{U}$ is a small category and $Y:\cat{U}\to \S$ is a functor, then the homotopy colimit of $Y$ can be expressed as the realization of the Reedy cofibrant simplicial object $\on{B}_\bullet(*,\cat{U},Y)$. From this construction it is clear that a functor preserves homotopy colimit if and only if it preserves homotopy colimits over $\Delta\op$ as well as coproducts.

Clearly $\int_{-}A$ send coproduct to coproducts in $\cat{C}[\oCom]$. 

Let $X_\bullet$ be a simplicial space. For each finite set $S$, we have an isomorphism $|X_\bullet|^S\cong |X^S_\bullet|$ because the realization of a simplicial space is just its diagonal. Hence we have

\begin{align*}
\int_{|X|}A&\cong \Map(-,|X|)\otimes_{\cat{Fin}}A\\
           &\cong |\Map(-,X_\bullet)|\otimes_{\cat{Fin}}A\\
           &\cong |\Map(-,X_\bullet)\otimes_{\cat{Fin}}A|\\
           &\cong |\int_{X_\bullet}A|
\end{align*}
In the last line the object $|\int_{X_\bullet}A|$ is a priori the geometric realization in $\cat{C}$ but this turns out to coincide with the geometric realization in $\cat{C}[\oCom]$ as is proved in the case of EKMM spectra in \cite[Proposition 2.3.]{brunmultiplicative}.
\end{proof}

\subsection{The commutative field theory}

This subsection is a toy-example of what we are going to consider in the sixth section. 

If $X$ is a space, we denote by $\S^{X/}$, the category of simplicial sets under $X$ with the model structure whose cofibrations, fibrations and weak equivalences are reflected by the forgetful functor $\S^{X/}\to \S$.

We define a large bicategory $\cat{Cospan}(\S)$. Its objects are the objects $\S$. 

The morphisms category $\cat{Cospan}(\S)(X,Y)$ is the category whose objects are diagrams of cofibrations:
\[Y\rightarrow U\leftarrow X\]
and whose morphisms are commutative diagrams:
\[
\xymatrix{
 &U\ar[dd]^{\simeq}& \\
X\ar[ur]\ar[dr]& &Y\ar[ul]\ar[dl]\\
 &V& 
}
\]
whose middle arrow is a weak equivalence.

The composition:
\[\cat{Cospan}(\S)(Y,Z)\times \cat{Cospan}(\S)(X,Y)\to \cat{Cospan}(\S)(X,Z)\]
is deduced from the Quillen bifunctor:
\[\S^{Z\sqcup Y}\times \S^{Y\sqcup X}\to \S^{Z\sqcup X}\]
taking $(Z\rightarrow A\leftarrow Y,Y\rightarrow B\leftarrow X)$ to $Z\rightarrow A\sqcup^Y B\leftarrow X$.

The bicategory $\cat{Cospan}(\S)$ is the underlying bicategory of a bioperad $\oper{C}\mathrm{ospan}(\S)$ which we now define.

\begin{defi}
A \emph{multi-cospan} from  $\{X_i\}_{i\in I}$ to $Y$ is a diagram:
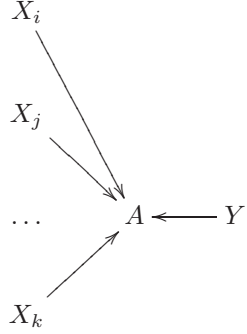
\[
\xymatrix{
X_i\ar[ddr]& & \\
X_j\ar[dr]& & \\
\ldots &A&Y\ar[l]\\
X_k\ar[ur]& & 
}
\]
where all the objects $X_i$ for $i\in I$ appear on the left of the diagram.
\end{defi}

There is a model category on the category of multi-cospans from $\{X_i\}_{i\in I}$ to $Y$ in which weak equivalences, fibrations and cofibrations are reflected by the forgetful functor to $\S$.

The category of multi-morphisms from $\{X_i\}_{i\in I}$ to $Y$ in the  bioperad $\oCospan(\S)$ is the category of weak equivalences between cofibrant multi-cospans from $\{X_i\}_{i\in I}$ to $Y$.

\begin{theo}\label{from cospan to modcat}
Let $A$ be a cofibrant commutative algebra in $\cat{C}$. There is a morphism of operad $\oCospan(\S)\to\oMod\oper{C}\mathrm{at}$ sending $X$ to $\Mod_{\int_X A}$.
\end{theo}

\begin{proof}
Let us first construct a morphism of operad 
\[\oCospan(\S)\to \oMor(\oCom)\]

We do this by sending the object $X$ to the right $\oCom$-module $\Map(-,X)$. We observe that $\Map(-,X)$ is a commutative algebra in $\Mod_{\oCom}$ and any map of simplicial sets $X\to Y$ induces a commutative algebra map $\Map(-,X)\to \Map(-,Y)$ making $\Map(-,Y)$ into a left module over $\Map(-,X)$. This observation implies that any multicospan from $\{X_i\}_{i\in I}$ to $Y$ represents an object of ${}_{\{\Map(-,X_i)\}_{i\in I}}\Mod_{\Map(-,Y)}$.

Moreover observe that if $X\leftarrow U\rightarrow Y$ is a diagram in $\S$ in which both maps are cofibrations, then the functor on finite sets $\Map(-,X\sqcup ^U Y)$ is isomorphic (not just weakly equivalent) to the functor $\Map(-,X)\otimes_{\Map(-,U)}\Map(-,Y)$. Indeed, both functors can be identified with the following functor:
\[S\mapsto \bigsqcup _{S=A\cup B}\Map((A,A\cap B),(X,U))\times_{\Map(A\cap B,U)}\Map((B,A\cap B),(Y,U))\]
This proves that the assignment $X\mapsto \Map(-,X)$ is a morphism of operads from $\oper{C}\mathrm{ospan}(\S)$ to $\oMor(\oCom)$.

We have already constructed a morphism of operad from $\oMor(\oCom)$ to $\oModCat$ in the third section. We can compose it with the map we have just constructed.
\end{proof}

\begin{rem}
In \cite{toenoperations}, To\"en proves that the operad $\oper{E}_d$ maps to the operad of endomorphisms of $S^{d-1}$ in $\oCospan(\S)$. In particular, if $A$ is a commutative algebra in $\cat{C}$, the category $\Mod_{\int_{S^{d-1}}A}$ is an $\oper{E}_d$-monoidal model category. In the following section, we generalize this result to the case where $A$ is an $\oper{E}_d$-algebra.
\end{rem}

\section{The field theory associated to an $\oper{E}_d$-algebra}

In this section we continue our study of the categories of modules over $\oper{E}_d$-algebras. More precisely, we study the operations between categories of modules built out of a cobordisms.

\subsection{The cobordism category}

\begin{cons}
Let $V$ be a $(d-1)$-dimensional real vector space and $\tau$ be a basis of $V\oplus\mathbb{R}$. We define by $-\tau$ the basis of $V\oplus\mathbb{R}$ which is the image of $\tau$ under the unique linear transformation of $V\oplus\mathbb{R}$ whose restriction to $V$ is the identity and whose restriction to $\mathbb{R}$ is the opposite of the identity.

More generally, if $S$ is a $(d-1)$-manifold and $\tau$ is a $d$-framing, we denote by $-\tau$ the $d$-framing obtained by applying the above procedure fiberwise in $TS\oplus\mathbb{R}$.
\end{cons}

\begin{defi}
Let $S_\sigma$ and $T_\tau$ be two $(d-1)$-manifold with a $d$-framing. A \emph{bordism from $S_{\sigma}$ to $T_{\tau}$} is a $d$-manifold $W$ with boundary together with the data of
\begin{itemize}
\item An embedding $\psi_{in}^W:S_\sigma\times[0,\epsilon)\to W$ 
\item An embedding $\psi_{out}^W:T_\tau\times(0,\epsilon]\to W$ 
\item A framing which restricts to $\sigma$ and $\tau$ on $S_\sigma\times[0,\epsilon)$ and $T\times(0,\epsilon]$.
\end{itemize}
This data is moreover require to have the following properties
\begin{itemize}
\item The embedding $\psi_{out}^W$ admits a smooth extension $\psi_{out}^W:T\times[0,\epsilon]\to W$ (which may not send $T\times\{0\}$ to the boundary of $W$).
\item The induced map $\psi_{in}\sqcup\psi_{out}:S\sqcup T\to\partial W$ is a diffeomorphism.
\end{itemize}
\end{defi}

\begin{rem}
Note that according to our conventions, the framing of $W$ looks like $\sigma$ in a neighborhood of $S\subset \partial W$ but looks like $-\tau$ in a neighborhood of $T\subset \partial W$. 
\end{rem}

If $W$ is such a bordism, we denote by $W^-$ the manifold $W-T\times\{\epsilon\}$ If $W$ and $W'$ are two bordisms from $S_\sigma$ to $T_\tau$, we denote by $\Emb_f^{S_\sigma,T_{\tau}}(W,W')$, the topological space $\Emb_f^{S_\sigma\sqcup T_{-\tau}}(W ,W')$. We denote by $\on{Diff}_f^{S_\sigma,T_{\tau}}(W,W')$ the subspace consisting of those embeddings which are surjective.

\begin{rem}
If $V$ is a bordism from $S_\sigma$ to $T_\tau$ and $W$ is a bordism from $T_\tau$ to $U_\upsilon$, we denote by $W\circ V$ the manifold $V^-\sqcup^{T\times[0,\epsilon)} W$. Where the map $T\times[0,\epsilon)\to W$ is the map $\psi_{in}^W$ and the map $T\times[0,\epsilon)\to V^-$ is the map $\psi_{out}^V$. Note that $W\circ V$ is in a canonical way a bordism from $S_\sigma$ to $U_\upsilon$. 

Notice that there is the structure of a self-bordism of $S_\sigma$ on the manifold $S\times[0,\epsilon]$ in which $\psi_{in}$ and $\psi_{out}$ are the obvious inclusions. Note that also that if $W$ is a bordism from $S_\sigma$ to $T_\tau$, then $W\circ S\times[0,\epsilon]$ and $T\times[0,\epsilon]\circ W$ are canonically diffeomorphic to $W$.
\end{rem}

\begin{cons}
We can construct a $2$-category enriched in spaces $f\cat{Cob}_d$ as follows. 

Let us fix a set $\mathsf{X}$ of closed $d-1$-manifolds with a $d$ framing containing at least one element of each diffeomorphism class of connected $d$-framed $(d-1)$-manifold. For $\{S_i\}$ a finite family of elements of $\mathsf{X}$ and $\{T_j\}$ a finite family of elements of $\mathsf{X}$, we fix $\mathsf{Y}(\{S_i\};\{T_j\})$ a set containing at least one element in each diffeomorphism class of pairs $(W,s)$ where $W$ is a $d$-manifold with boundary and $s$ is the data of a the structure of a bordism from $\bigsqcup_iS_i$ to $\bigsqcup_jT_j$ on $W$.

We define our $2$-category to have as objects the finite families of elements of $\mathsf{X}$ and as morphism from $\{S_i\}$ to $\{T_j\}$ the sequences of bordisms $(W_0,W_1,\ldots W_n)$ such that for each $i$, the source of $W_i$ coincides with the target of $W_{i-1}$ and such that the source of $W_0$ is $\{S_i\}$ and the target of $W_{n}$ is $\{T_j\}$. The space of $2$-morphisms from $(W_0,\ldots, W_n)$ to $(W'_0,\ldots, W'_m)$ is the group \[\on{Diff}^{\sqcup S_i,\sqcup T_j}_f(W_n\circ\ldots W_0,W'_m\circ\ldots W'_0)\]
There is a symmetric monoidal structure on $f\cat{Cob}_d$ sending $\{S_i\}_{i\in I}$ and $\{S_j\}_{j\in J}$ to $\{S_i\}_{i\in I\sqcup J}$. 

We can apply the functor nerve to each Hom groupoid and we find a simplicial category also denoted $f\cat{Cob}_d$

Note that changing the sets $\mathsf{X}$ and $\mathsf{Y}$ does not change the $2$-category $f\cat{Cob}_d$ up to equivalence of categories.
\end{cons}

\begin{rem}
Note that in our version of the cobordism category, both the objects and the morphisms are allowed to be non-compact.
\end{rem}

\subsection{Construction of a bimodule from a bordism}

If $\cat{C}$ is a symmetric monoidal category and $\cat{X}$ is tensored over $\cat{C}$ and $A$ is an $\oper{O}$-algebra in $\cat{C}$, we can define the category of $P$-shaped $A$-modules in $\cat{X}$ as the category of $U^P_A$-modules in $\cat{X}$.

In particular, let $Q$ be an associative algebra in $\Mod_{\oper{O}}$. The category $\Mod_Q$ is tensored over $\Mod_{\oper{O}}$. The operad $\oper{O}$ seen as a right module over itself is an $\oper{O}$-algebra in $\Mod_{\oper{O}}$. Thus it makes sense to speak of $P$-shaped modules over $\oper{O}$ in the category $\Mod_Q$. It is straightforward to check that those are the same as $P$-$Q$-bimodules in $\Mod_{\oper{O}}$.

This observation will be useful in the following construction.

\begin{cons}
Let $W$ be an $S_\sigma$-manifold. We define $\und{W}$, a right module over $\oper{E}_d$ with a left $\und{S}_\sigma$-action.

As a right module over $\oper{E}_d$, $\und{W}$ sends $D^{\sqcup n}$ to $\Emb^{S_\sigma}(S\times[0,1)\sqcup D^{\sqcup n},W)$. The left $\und{S}_\sigma$-module structure comes from the composition
\[\Emb^{S_\sigma}(S\times[0,1)\sqcup D^{\sqcup p}, S\times[0,1))\times\Emb^{S_\sigma}(S\times[0,1)\sqcup D^{\sqcup n},W)\to \Emb^{S_\sigma}(S\times[0,1)\sqcup D^{\sqcup n+p},W)\]

Note that the construction $W\mapsto \und{W}$ is functorial in $W$. More precisely for $W$ and $W'$ two $S_\sigma$-manifolds, there is a map
\[\Emb^{S_\sigma}(W,W')\to \Map_{L\Mod_{\und{S}_\sigma}}(\und{W},\und{W}')\]

Now, let $W$ be a bordism from $S_\sigma$ to $T_\tau$. Let $\widetilde{W}$ be the $S_\sigma$-manifold $T\times[0,1)\circ W$. We have the left $S_\sigma$-module $\undtil{W}$. Observe that there is a map
\[\Emb^{T_\tau}_f(T\times[0,1)\sqcup D^{\sqcup n},T\times[0,1))\to \Emb^{T_\tau}_f(\widetilde{W}\sqcup D^{\sqcup n},\widetilde{W})\]
obtained by extending embedding by the identity on $W\subset \widetilde{W}$. 

This implies that there are maps
\[\Emb^{T_\tau}(T\times[0,1)\sqcup D^{\sqcup n},T\times[0,1))\to\Map_{\Mod_{\und{S}_\sigma}}(\und{\widetilde{W}\sqcup D^{\sqcup n}},\undtil{W})\]

Note moreover that as right modules over $\oper{E}_d$ with left $\und{S}_\sigma$-module structure, we have an isomorphism
\[\und{\widetilde{W}\sqcup D^{\sqcup n}}\cong \undtil{W}\otimes \und{D}^{\otimes n}\]

These two facts together imply that $\undtil{W}$ is an $\und{S}_\sigma$-$\und{T}_\tau$-bimodule in the category $\Mod_{\oper{E}_d}$ 
\end{cons}

\begin{rem}
The above construction relies on the following two observations:
\begin{itemize}
\item The category of $S_\sigma$ manifold is tensored over the category of framed $d$-manifold.
\item The manifold $\widetilde{W}$ is a $T_\tau$-shaped module over $D$ in the category of $S_\sigma$-manifolds.
\end{itemize}
\end{rem}

\begin{cons}\label{functor induced by a bordism}
Let $A$ be a cofibrant $\oper{E}_d$-algebra in $\cat{C}$. Let $W$ be a bordism from $\und{S}_\sigma$ to $\und{T}_\tau$, we define a functor $P_W$
\[P_W:\und{S}_\sigma\Mod_A\to \und{T}_\tau\Mod_A\]
by sending $M$ to $\undtil{W}\otimes_{\und{S}_\sigma\Mod}(A,M)$
\end{cons}

Recall that $P_W(-)$ can be derived by restricting it to cofibrant $\und{S}_\sigma$-shaped modules and we have an equivalence
\[\int_{\widetilde{W}}(A,M)\cong \L P_W(M)\]
by definition of factorization homology over an $S_\sigma$-manifold. 

\subsection{Compositions}

We now want to study composition of functors of the form $P_W$.

\begin{prop}\label{lax functor}
Let $W$ be a bordism from $S_{\sigma}$ to $T_{\tau}$ and $W'$ be a bordism from $T_{\tau}$ to $U_{\upsilon}$, then there is a map
\[\undtil{W}\otimes_{\und{T}}\undtil{W'}\to\undtil{W'\circ W}\]
\end{prop}

\begin{proof}
We first construct a map $u:\undtil{W}\otimes\undtil{W'}\to \undtil{W'\circ W}$. Let 
\[\phi:S\times[0,1)\sqcup D^{\sqcup p}\to \widetilde{W}\]
and 
\[\phi':T\times[0,1)\sqcup D^{\sqcup p'}\to \widetilde{W'}\]
Then $u(\phi,\psi)$ is constructed as follows. First, we extend $\phi'$ to an embedding 
\[\phi'':\widetilde{W}\sqcup D^{\sqcup p'}\to \widetilde{W'\circ W}\]
by glueing a copy of $W$ on $\phi'$. Then we consider the composite of $\phi''$ with $\phi$ and find an embedding
\[u(\phi,\phi')=\phi''\circ \phi:S\times[0,1)\sqcup D^{\sqcup p}\sqcup D^{\sqcup p'}\to \widetilde{W'\circ W}\]

We now want to show that this map factors through $\undtil{W}\otimes_{\und{T}}\undtil{W'}$. It suffices to show that for any $m$ and any embedding $\psi$
\[\psi:T\times[0,1)\sqcup D^{\sqcup m}\to T\times[0,1)\]
there is an equality in $\undtil{W'\circ W}(p+m+p')$:
\[u(\phi.\psi,\phi')= u(\phi,\psi.\phi')\]
where $\phi.\psi\in\undtil{W}(p+m)$ is obtained by the right action of $\und{T}$ on $\undtil{W}$ and $\psi.\phi'\in\undtil{W}'(m+p')$ is obtained by the left action of $\und{T}$ on $\undtil{W}'$.

But the above equality can be checked explicitly.
\end{proof}

\begin{prop}\label{composition of bordisms}
Let $W$ be a bordism from $S_{\sigma}$ to $T_{\tau}$ and $W'$ be a bordism from $T_{\tau}$ to $U_{\upsilon}$. Let $M$ be a $\und{S}_\sigma$-module, then there is a weak equivalence:
\[\L P_{W'}(\L P_W(M))\simeq \L P_{W'\circ W}(M)\]
\end{prop}

\begin{proof}
First notice, that $P_W$ sends cofibrant modules to cofibrant modules, therefore, we can assume that $M$ is cofibrant and prove that $P_{W'}\circ P_W(M)\simeq P_{W'\circ W}(M)$.

According to \ref{nb2}. We have 
\[P_{W'}(P_W(M))\simeq\on{hocolim}_{\cat{D}(\widetilde{W'})}(A,P_W(M))\]

Let $\cat{E}$ be the category of open sets of $\widetilde{W'\circ W}$ of the form $Z\sqcup D^{\sqcup n}$ where $Z$ is a submanifold of $\widetilde{W'\circ W}$ which contains $W$ and which is such that there is a diffeomorphism $Z\cong \widetilde{W}$ inducing the identity on $W$. In other words, $Z$ is $W$ together with a collar of the $T$ boundary which is contained in the $W'$ side.

We claim that 
\[P_{W'\circ W}(M)\simeq\on{hocolim}_{E\in \cat{E}} \int_E(A,M)\]
The proof of this claim is entirely analogous to \ref{nb2}. 

If $E$ is of the form $Z\sqcup D^{\sqcup n}$ and $Z$ is as in the previous paragraph, we have $\int_E(A,M)\cong P_W(M)\otimes A^{\otimes n}$. Moreover, the category $E$ is isomorphic to $\cat{D}(\widetilde{W'})$ under the map sending $E$ to the intersection of $E$ with the $W'$ half of $\widetilde{W'\circ W}$.

Thus, we have identified both $P_W'\circ P_W(M)$ and $P_{W'\circ W}(M)$ with the same homotopy colimit.
\end{proof}

The cobordism category $f\cat{Cob}_d$ has an approximation in the world of right modules over $\oper{E}_d$ that we now describe.

\begin{defi}
Let $\und{S}_\sigma$ and $\und{T}_\tau$ be two $d$-framed, $(d-1)$-manifolds. A $\und{S}_\sigma$-$\und{T}_\tau$-bimodule is called \emph{representable} if it is isomorphic to $\undtil{W}$ for some bordism $W$ from $\und{S}_\sigma$ to $\und{T}_\tau$.
\end{defi}

\begin{cons}\label{fake cobordism category}
We construct a bioperad $f\widehat{\oper{C}ob}_d$.

Its objects are $(d-1)$-manifolds with a $d$-framing. For $\{(S_\sigma)_{i\in I}\}_i$ a finite collection of $d$-framed $(d-1)$-manifolds and $T_\tau$ a $d$-framed $(d-1)$-manifold, we define $f\widehat{\oper{C}ob}_d(\{(S_\sigma)_i\}_{i\in I},T_\tau)$ to be full subcategory of the category 
\[\oMor(\oper{E}_d)(\{(\und{S}_\sigma)_i\}_{i\in I},\und{T}_\tau)\]
on objects that are weakly equivalent to representable bimodules.

The composition comes from the composition in $\oMor(\oper{E}_d)$. The fact that the composition of two representable bimodules is equivalent to a representable bimodules follows from \ref{composition of bordisms}.
\end{cons}

Now let us compare $f\cat{Cob}_d$ and $f\widehat{\cat{Cob}}_d$. In the two categories the objects are the same, namely $(d-1)$-manifolds with a $d$-framing. In $f\cat{Cob}_d$, the space of maps from $\und{S}_\sigma$ to $\und{T}_\tau$ is equivalent to
\[\bigsqcup_V B\on{Diff}_f^{S\sqcup T}(V)\]
where the disjoint union is taken over all diffeomorphism classes of bordisms.

In $f\widehat{\cat{Cob}}_d$, the nerve of the category of maps from $\und{S}_\sigma$ to $\und{T}_\tau$ is equivalent to
\[\bigsqcup_V B\on{Auth}_{S_\sigma-T_\tau}(\Emb^S_f(-,\widetilde{V}))\]
where the disjoint union is taken over the set of equivalence classes of representable $\und{S}_\sigma$-$\und{T}_\tau$-bimodules and the homotopy automorphisms are taken in the model category of $\und{S}_\sigma$-$\und{T}_\tau$-bimodules. There is an obvious map
\[\on{Diff}_f^{S\sqcup T}(V)\to\on{Auth}_{S_\sigma-T_\tau}(\Emb^S_f(-,\widetilde{V}))\]
which can be interpreted as the map from the group of diffeomorphisms to the limit of its embedding calculus (see section 9 of \cite{boavidamanifold}). In that sense, the bicategory $f\widehat{\cat{Cob}}_d$ is an embedding calculus approximation of the cobordism category. 

\medskip
In proposition \ref{lax functor}, we have essentially constructed a lax functor
\[f\oper{C}ob_d\to f\widehat{\oper{C}ob_d}\]
which is unfortunately not a pseudo-functor. However proposition \ref{composition of bordisms} shows that the coherence $2$-morphisms for this lax-functor are equivalence. We believe that this gives enough structure to strictify this map into an actual map from some version of the cobordism operad to the Morita operad of $\oper{E}_d$. One should observe that, although it is less geometric than the cobordism category, the Morita operad of $\oper{E}_d$ is more approachable to computations. In fact, the philosophy of embedding calculus is exactly to replace manifold by the functor they represent on the category of finite disjoint unions of disks. It can also happen that the embedding calculus converges to something more interesting than the actual space of embeddings. For instance the group of framed diffeomorphisms of the disk fixing the boundary is contractible. On the other hand, its embedding calculus approximation (when working over the rationals) contains the Grothendieck-Teichm\"uller Lie algebra (see \cite{aronegraph} and \cite{willwachermaxim}).

To conclude, we can prove the main theorem of this section.

\begin{theo}
Let $A$ be a cofibrant $\oper{E}_d$-algebra in $\cat{C}$. There is a map of operad
\[f\widehat{\oper{C}ob}_d\to\oModCat\]
mapping the object $\und{S}_\sigma$ to $\und{S}_\sigma\Mod_A$.
\end{theo}

\begin{proof}
We already know what this map is on objects. Let $W$ be an object of $\cat{Cob}(\und{S}_\sigma,\und{T}_\tau)$, we have the functor $P_W$ from $\und{S}_\sigma\Mod_A$ to $\und{T}_\tau\Mod_A$  which defines a functor $f\widehat{\cat{Cob}}_d\to\cat{ModCat}$. Extending this to a map of bioperad is straightforward and then we apply \ref{strictification}.
\end{proof}

\begin{rem}
It seems to be a folk theorem that the suboperad of $f\oCob_d$ on the object $S^{d-1}_\kappa$ receives a map from $\oper{E}_d$. If we combine this result with the previous proposition, we recover the fact that the category of operadic $\oper{E}_d$-modules has an action by $\oper{E}_d$.
\end{rem}

\subsection{Change of dimension}

We have maps $\alpha_k:\oper{E}_d\to\oper{E}_{d+k}$ defined for any $k\geq 0$ which are obtained by taking the product of a framed embedding $D^{\sqcup k}\to D$ with $\mathbb{R}^k$.

In particular using \ref{change of operad}, we see that, for $S_\tau$ a $d$-framed $(d-1)$-manifold, there is an associative algebra in right module $(\alpha_k)_!\und{S}_\tau$ over $\oper{E}_{d+k}$ which parametrizes a certain shape of modules. We use the notation $\und{S}^{(k)}_\tau$ for this object.

\begin{prop}
The associative algebra in $\Mod_{\oper{E}_d}$ $\und{S}_\tau^{(k)}$ is equivalent to $\und{(S\times\mathbb{R}^k)}_\tau$.
\end{prop}

\begin{proof}
We compute 
\[\und{S}^{(k)}(n)=\Emb^S_f(S\times[0,1)\sqcup -,S\times[0,1))\otimes^{\L}_{\cat{E}_d}\Emb_f((D\times\mathbb{R}^k)^{\sqcup n},-\times\mathbb{R}^k)\]
One shows exactly as in \ref{nb2} that this computation can be reduced to computing
\[\on{hocolim}_{U\in\cat{D}(S)}\Emb_f^S(S\times[0,1)\times\mathbb{R}^k\sqcup(D\times\mathbb{R}^k)^{\sqcup n},U\times\mathbb{R}^k)\]
which is equivalent to $\Emb_f^S(S\times[0,1)\times\mathbb{R}^k\sqcup(D\times\mathbb{R}^k)^{\sqcup n},S\times[0,1))$.
\end{proof}

\begin{rem} As a particular case of this construction, we can look at the map $\oper{E}_1\to\oper{E}_d$. We have the right module $L$ obtained from the point (seen as a zero manifold) with one of the two possible $1$-framing so that $L\Mod_A$ is a model for left modules over the $\oper{E}_1$-algebra $A$.

The previous proposition tells us that for an $\oper{E}_d$-algebra $A$, the category of left module over the underlying $\oper{E}_1$-algebra is equivalent to the category of $\mathbb{R}^{d-1}$-shapes modules. 

We believe that similarly to the fact that $\oper{E}_d$ maps to the endomorphisms of $S^{d-1}_\kappa$ in $f\oCob_d$, the operad $\oper{E}_{d-1}$ maps to the operad of endomorphisms of $\mathbb{R}^{d-1}$ in $f\oCob_d$.

Even more generally, we believe that the swiss-cheese operad of Kontsevich and Voronov (see \cite{voronovswiss} or \cite{horelfactorization}) maps to the endomorphisms operad of the pair $(S^{d-1}_\kappa,\mathbb{R}^{d-1})$. This result, if true would imply that the operad $\oper{SC}_d$ of $d$-dimensional swiss-cheese acts on the pair $(\und{S}^{d-1}_\kappa\Mod_A,\und{\mathbb{R}}^{d-1}\Mod_A)$. This result is closely related to Kontsevich's version of Deligne's conjecture. We hope to prove this missing statements in future work.
\end{rem}

\appendix

\section{The homotopy type of certain spaces of embeddings}

In this appendix, we collect a few facts about the homotopy type of certain spaces of embeddings.

We denote by $D$, the $d$-dimensional euclidean space $\mathbb{R}^d$.

\begin{prop}\label{framed embeddings of disks}
Let $M$ be a framed manifold with boundary. The evaluation at the center of the disks induces a weak equivalence
\[\Emb_f(D^{\sqcup p},M)\to\on{Conf}(p,M-\partial M)\]
\end{prop}

\begin{proof}
See \cite[proposition 6.6]{horelfactorization}.
\end{proof}

Now we want to study the spaces $\Emb^S(M,N)$ and $\Emb^{S_\tau}_f(M,N)$. Note that the manifold $S\times[0,1)$ is canonically an $S$-manifold if we take the map $S\times[0,\epsilon)\to S\times[0,1)$ to be the inclusion.

If $\tau$ is a framing of $TS\oplus\mathbb{R}$, then $S\times[0,1)$ is a framed $S_\tau$-manifold in a canonical way.

\begin{prop}
Let $M$ be an $S$-manifold. The space $\Emb^S(S\times[0,1),M)$ is contractible. 
\end{prop}

\begin{proof}\footnote{We wish to thank Martin Palmer for helping us with this proof.} 
If $S$ is compact, this is proved in \cite[proposition 6.7]{horelfactorization}. 

We denote by $E$ the space $\Emb^S(S\times[0,1),M)$. Let us pick a collar $C$ of the boundary, i.e. $C$ is the image of an embedding $S\times[0,1)\to M$.

(1) We start with an easy observation. 

Let $X$ be a topological space and $\{f_i\}_{i\in I}$ be a finite family of continuous functions on $X$ taking values in $[0,\infty)$. Then the function
\[x\mapsto \on{sup}_{i\in I} f_i(x)\]
is continuous.

Since continuity is a local property, the above conclusion remains true if we drop the finiteness assumption but instead assume that, locally around each point, only a finite number of the functions are non-zero.

(2) Let $K$ be a compact subset of $S$. We denote by $U_{K,n}$ the open set in $E$ containing the embeddings $e$ such that $e(K\times[0,\frac{1}{n}])\subset C$. The open sets $U_{K,n}$ form an open cover of the paracompact space $E$. Let $\phi_{K,n}$ be a partition of unity subordinate to that open cover.

We define a function
\[f_K:E\to [0,1)\]
sending $e$ to $\on{sup}_{n}\frac{1}{n}\phi_{K,n}(e)$.

This function is nowhere $0$ and continuous by claim (1). Moreover, for any $e\in E$, and any $s\in K$, we see that $e$ sends $\{s\}\times[0,f_K(s))$ to a subset of $C$.

(3) Now we pick $\mathcal{K}$ a cover of $S$ by compact subsets whose interiors form an open cover (one can for instance take the collection of all compact subsets). Let $\phi_K$ be a smooth partition of unity subordinate to the open cover $\{\on{int}(K)\}_{K\in\mathcal{K}}$.

We define a function
\[f:E\times S\to[0,1)\]
sending $(e,s)$ to $\sum_{K\in \mathcal{K}}\phi_K(s)f_K(e)$. This function takes values in $[0,1)$ and is continuous. This function is also clearly nowhere zero. Indeed, for each $s$, there is $K$ such that $\phi_K(s)\neq 0$ and for that $K$, $f_K(e)\neq 0$ according to (2). 

(4) Let us fix a point $s\in S$ and $e\in E$. The value of $f(e,s)$ is a weighted sum of the $f_K(e)$ for $\phi_K(s)\neq 0$. We observe that the total weight is exactly $1$ which implies that
\[f(e,s)\leq \on{sup}_{K,\phi_K(s)\neq 0}f_K(e)\]

Let $K_0$ be a compact which realizes this sup ($K_0$ exists by definition of a partition of unity). Hence the image of $s\times[0,f_{K_0}(e))$  is contained in $C$ by (2) and contains the image of $s\times[0,f(e,s))$ under $e$. 

In conclusion, $f$ is a smooth function $E\times S\to [0,1)$ which is nowhere $0$ and which is such that the image of $s\times[0,f(e,s))$ under $e$ is contained in $C$.

(5) Now we use the function $f$ to construct a map from $E$ onto the subspace $\Emb^S(S\times[0,1),C)$. We define 
\[\kappa:E\to \Emb^S(S\times[0,1),C)\]
by the formula $\kappa(e)(s,t)=e(s,f(e,s)t)$. The function $f$ has been designed exactly  to insure that the image of $\kappa(e)$ is contained in $C$. 

(6) Let $\iota:\Emb^S(S\times[0,1),C)\to E$ be the inclusion. We claim that $\iota\circ\kappa$ is homotopic to $\id_E$. Indeed, we have the homotopy
\[H:E\times[0,1]\to E\]
defined by $H(e,u)(s,t)=e(s,t[(1-u)+uf(e,s)])$. Then $H(e,0)=e$ and $H(e,1)=\kappa(e)$. One would show similarly that $\kappa\circ\iota$ is homotopic to $\id_{\Emb^S(S\times[0,1),C)}$.

(7) Now we are reduced to proving the contractibility of $\Emb^S(S\times[0,1),S\times[0,1))$. But in \cite[proposition 6.7]{horelfactorization} we give a proof of that fact which does not depend on the compactness of $S$.
\end{proof}

\begin{prop}\label{contractibility collars framed}
Let $N$ be a framed $S_\tau$-manifold. The space $\Emb_f^{S_\tau}(S\times[0,1),N)$ is contractible.
\end{prop}

\begin{proof}
The proof uses the previous proposition exactly as in \cite[proposition 6.8]{horelfactorization}.
\end{proof}

\section{A few facts about model categories}

\subsection{Monoidal and enriched model categories}

\begin{defi}
Let $\cat{X}$, $\cat{Y}$ and $\cat{Z}$ be three model categories.
A pairing $T:\cat{X}\times\cat{Y}\to\cat{Z}$ is said to satisfies the \emph{pushout-product axiom} if for each pair of cofibrations $f:A\to B$ of $\cat{X}$ and $g:K\to L$ of $\cat{Y}$, the induced map
\[T(B,K)\sqcup^{T(A,K)}T(A,L)\to T(B,L)\]
is a cofibration which is trivial if one of $f$ and $g$ is.

We say that $T$ is a \emph{left Quillen bifunctor} if it satisfies the pushout-product axiom and if it is a left adjoint when one variable is fixed.
\end{defi}

One useful consequence of the pushout-product axiom is that if $A$ is cofibrant $T(A,-)$ preserves trivial cofibrations between cofibrant objects. Then by Ken Brown's lemma (see \cite{hoveymodel}) it preserves all weak equivalences between cofibrant objects.

Recall that if $\cat{X}$ is a model category, $\cat{X}\op$ has a dual model structure in which (trivial) fibrations are opposite of (trivial) cofibrations and weak equivalences are opposite of weak equivalences.

\begin{defi}
A \emph{(closed) monoidal model} category is a model category structure on a (closed) monoidal category $(\cat{V},\otimes,\un)$ which is such that
\begin{itemize}
\item The functor $-\otimes -:\cat{V}\times\cat{V}\to\cat{V}$ satisfies the pushout-product axiom.
\item The map $Q\un\to\un$ induces a weak equivalence $Q\un\otimes V\to V$ for each $V$.
\end{itemize}
A symmetric monoidal model category is a model category structure on a symmetric monoidal category which makes the underlying monoidal category into a monoidal model category.
\end{defi}

\begin{defi}
Let $\cat{V}$ be a monoidal model category. Let $(\cat{X},\Hom_{\cat{X}}(-,-))$ be a $\cat{V}$-enriched category. A \emph{$\cat{V}$-enriched model structure} on $\cat{X}$ is a model category structure on the underlying category of $\cat{X}$ that is such that the functor
\[\Hom_{\cat{X}}\op:\cat{X}\times\cat{X}\op\to\cat{V}\op\]
is a left Quillen bifunctor.
\end{defi}

Note that in a $\cat{V}$-enriched model category $\cat{X}$, we have a tensor and cotensor functor:
\[\cat{V}\times\cat{X}\to\cat{X},\;\;\;\;\cat{V}\op\times\cat{X}\to\cat{X}\]
fitting into the usual two variables adjunction.

\begin{defi}
Let $(\cat{X},\Hom_{\cat{X}})$ be a $\cat{V}$-enriched category. Let $T$ be a monad on $\cat{X}$ and $\cat{X}[T]$ be the category of $T$-algebras in $cat{X}$. Let us define the following equalizer
\[\Hom_{\cat{X}[T]}(X,Y)\to\Hom_\cat{X}(X,Y)\rightrightarrows\Hom_{\cat{X}}(TX,Y)\]
where the top map is obtained by precomposition with the structure map $TX\to X$ and the bottom map is the composition
\[\Hom_\cat{X}(X,Y)\to\Hom_\cat{X}(TX,TY)\to\Hom_\cat{X}(TX,Y)\]
\end{defi}

\begin{prop}\label{enriched monad}
Let $\cat{V}$ be a monoidal model category and $(\cat{X},\Hom_{\cat{X}})$ be a $\cat{V}$-enriched model category. Let $\cat{X}$ be a cofibrantly generated model category. If the category $\cat{X}[T]$ can be given the transferred model structure. Then $\cat{X}[T]$ equipped with $\Hom_{\cat{X}[T]}$ is a $\cat{V}$-enriched model category.
\end{prop}

\begin{proof}
Let $f:U\to V$ be a (trivial) cofibration and $p:X\to Y$ be a fibration in $\cat{C}[T]$. We want to show that the obvious map
\[\Hom_{\cat{X}[T]}(V,X)\to\Hom_{\cat{X}[T]}(U,X)\times_{\Hom_{\cat{X}[T]}(U,Y)}\Hom_{\cat{X}[T]}(V,Y)\]
is a (trivial) fibration in $\cat{V}$. It suffices to do it for all generating (trivial) cofibration $f$. Hence it suffices to do this for a free map $f=Tm:TA\to TB$ where $m$ is a (trivial) cofibration in $\cat{X}$. But then the statement reduces to proving that
\[\Hom_{\cat{C}}(B,X)\to\Hom_{\cat{C}}(A,X)\times_{\Hom_{\cat{C}}(A,Y)}\Hom_{\cat{C}}(B,Y)\]
is a (trivial) fibration which is true because $\cat{C}$ is a $\cat{V}$-enriched model category.
\end{proof}

\begin{prop}\label{model structure on modules}
Let $\cat{V}$ be a cofibrantly generated monoidal model category. Let $R$ be an associative algebra in $\cat{V}$ whose underlying object is cofibrant or $R$ be any associative algebra if $\cat{V}$ satisfies the monoid axiom. Then the transferred model structure on the category $\Mod_R$ of right $R$-modules in $\cat{V}$ exists. Moreover, if $\cat{V}$ is symmetric monoidal and $R$ is a commutative algebra, $\Mod_R$ is a symmetric monoidal model category for the relative tensor product $-\otimes_R-$. $\hfill\square$. 

Finally, if $\cat{V}$ is enriched  over a monoidal model category $\cat{W}$, then, so is $\Mod_R$.
\end{prop}

\begin{proof}
The existence of the model structure is straightforward and can be found in many sources (for instance \cite{schwedealgebras}). The fact about enrichments follows from \ref{enriched monad}.
\end{proof}

\begin{prop}\label{Quillen equivalence module}
If $f:R\to S$ is a map between associative algebras of $\cat{V}$ that are cofibrant in $\cat{V}$ or any associative algebras if $\cat{V}$ satisfies the monoid axiom, then the functor $f_!:\Mod_R\to \Mod_S$ sending $M$ to $M\otimes_R S$ is a Quillen left adjoint. 

Moreover, if $R$ and $S$ are cofibrant in $\cat{V}$ and if weak equivalences in $\cat{V}$ are preserved under filtered colimits, then, $f_!$ is a left Quillen equivalence if $f$ is a weak equivalence.
\end{prop}

\begin{proof}
The right adjoint of $f_!$ is the forgetful functor $f^*$ from $\Mod_S$ to $\Mod_R$ which obviously preserves fibrations and weak equivalences. Hence $f_!$ is left Quillen. 

Now, assume that $f$ is a weak equivalence. We want to show that $u_M:M\to M\otimes_RS$ is a weak equivalence if $M$ is cofibrant. Clearly this is true for $M$ of the form $X\otimes R$ with $X$ cofibrant. Now assume that $u_M$ is a weak equivalence for some $M$ and let $N$ be the pushout of $X\otimes R\to M$ along a map $i\otimes R:X\otimes R\to Y\otimes R$ where $i$ is a cofibration in $\cat{C}$. Then the map $u_N$ is the map
\[Y\otimes R\sqcup^{X\otimes R}M\to Y\otimes S\sqcup^{X\otimes S}M\otimes_RS\]

Since $i$ is a cofibration and $R$ and $S$ are $\cat{V}$-cofibrant, both pushouts are homotopy pushouts. Therefore the map $u_N$ is a weak equivalence. Finally since weak equivalences are preserved under filtered colimits by assumption and under retract (because this is the case in any model category), $u_M$ is a weak equivalence for any cofibrant object in $\Mod_R$
\end{proof}

\subsection{Homotopy colimits and bar construction}

See \cite{dwyerhomotopy} or \cite{shulmanhomotopy} for a general definition of derived functors. We will use the following

\begin{prop}
Let $\cat{X}$ be a model category tensored over $\S$ and $s\cat{X}$ be the category of simplicial objects in $\cat{X}$ with the Reedy model structure. Then the geometric realization functor
\[|-|:s\cat{X}\to\cat{X}\]
is left Quillen
\end{prop}

\begin{proof}
See \cite[VII.3.6.]{goersssimplicial}.
\end{proof}

\begin{prop}\label{Reedy cofibrant}
Let $\cat{X}$ be a simplicial model category, let $\cat{K}$ be a simplicial category and let $F:\cat{K}\to\cat{X}$ and $W:\cat{K}\op\to \S$ be simplicial functors. Then the Bar construction
\[\on{B}_{\bullet}(W,\cat{K},F)\]
is Reedy cofibrant if $F$ is objectwise cofibrant.
\end{prop}

\begin{proof}
See \cite{shulmanhomotopy}.
\end{proof}

\begin{defi}
Same notation as in the previous proposition. Assume that $\cat{X}$ has a simplicial cofibrant replacement functor $Q$. We denote by $W\otimes^{\L}_{\cat{K}}F$ the realization of the simplicial object
\[\on{B}_{\bullet}(W,\cat{K},F\circ Q)\]
\end{defi}

Finally let us mention the following proposition which insures that having a simplicial cofibrant replacement diagram is not a strong restriction

\begin{prop}
Let $\cat{X}$ be a cofibrantly generated simplicial model category. Then $\cat{X}$ has a simplicial cofibrant replacement functor.
\end{prop}

\begin{proof}
See \cite[theorem 6.1.]{blumberghomotopical}.
\end{proof}

\subsection{Model structure on symmetric spectra}

Let $E$ be a an associative algebra in symmetric spectra. Then $\Mod_E$ has (at least) two simplicial cofibrantly generated model category structures in which the weak equivalences are the stable equivalences of the underlying symmetric spectrum:
\begin{itemize}
\item The positive model structure that we denote $\Mod_E$.

\item The absolute model structure that we denote $\Mod_E^a$.
\end{itemize}

Moreover if $E$ is commutative, both are closed symmetric monoidal model categories.

The identity functor induces a Quillen equivalence
\[\Mod_E\leftrightarrows \Mod_E^a\]

Both model structures have their advantages. The absolute model structure has more cofibrant objects (for instance $E$ itself is cofibrant which is often convenient). On the other hand the positive model structure has fewer cofibrant objects but a very well-behaved monoidal structure.

\section{Operads and modules}

\subsection{Colored operad}

In this paper, we call operad a symmetric colored operads in simplicial sets (also called a multicategory). When we want to specifically talk about operads with only one object, we say ``one-object operad''. If $\oper{M}$ is an operad, we write
\[\oper{M}(\{m_i\}_{i\in I};n)\]
for the space of operations from the set of the $m_i$'s to $n$.

Recall that any symmetric monoidal category can be seen as an operad:

\begin{defi}
Let $(\cat{A},\otimes,\un_\cat{A})$ be a small symmetric monoidal category enriched in $\S$. Then $\cat{A}$ has an underlying  operad $\oper{U}\cat{A}$ whose objects are the objects of $A$ and whose spaces of operations are given by
\[\oper{U}\cat{A}(\{a_i\}_{i\in I};b)=\Map_{\cat{A}}(\bigotimes_{i\in I}a_i,b)\]
\end{defi}

The construction $\cat{A}\mapsto \oper{U}\cat{A}$ sending a symmetric monoidal category to an operad has a left adjoint. The underlying category of the left adjoint applied to $\oper{M}$ is $\cat{M}$. A construction of that left adjoint is given in the first section of \cite{horelfactorization}.

We define an algebra over an operad $\oper{M}$ with value in a symmetric monoidal category $(\cat{C},\otimes,\un_{\cat{C}})$ as a morphism of operad $\oper{M}\to\oper{U}\cat{C}$. Equivalently, an $\oper{M}$-algebra in $\cat{C}$ is a symmetric monoidal functor $\cat{M}\to\cat{C}$. We will use the same notation for the two objects and allow oursleves to switch between them without mentioning it.

\subsection{Right modules over operads}

\begin{defi}
Let $\oper{M}$ be an operad. A \emph{right $\oper{M}$-module} is a simplicial functor
\[R:\cat{M}\op\to \S\]

When $\oper{O}$ is a single-object operad, we denote by $\Mod_\oper{O}$ the category of right modules over $\oper{O}$.
\end{defi}

Let $\Sigma$ be the category whose objects are the finite sets $\{1,\ldots,n\}$ with $n\in\mathbb{Z}_{\geq 0}$ and morphisms are bijections. $\Sigma$ is a symmetric monoidal category for the disjoint union operation.

Let $\oper{I}$ be the initial one-object operad (i.e. $\oper{I}(1)=*$ and $\oper{I}(k)=\varnothing$ for $k\neq 1$). It is clear that the free symmetric monoidal category associated to $\oper{I}$ is the category $\Sigma$. Let $\oper{O}$ be an operad and $\cat{O}$ be the free symmetric monoidal category associated to $\oper{O}$. By functoriality of the free symmetric monoidal category construction, there is a symmetric monoidal functor $\Sigma\to \cat{O}$ which induces a functor
\[\on{Fun}(\cat{O}\op,\S)\to\on{Fun}(\Sigma\op,\S)\] 

Recall the definition of the Day tensor product:

\begin{defi}
Let $(\cat{A},\square,\un_\cat{A})$ be a small symmetric monoidal category, then the category $\on{Fun}(\cat{A},\S)$ is a symmetric monoidal category for the operation $\otimes$ defined as the following coend:
\[F\otimes G(a)=\cat{A}(-\square-,a)\otimes_{\cat{A}\times\cat{A}}F(-)\times G(-)\]
\end{defi}

Now we can make the following proposition:

\begin{prop}
Let $\oper{O}$ be a single-object operad. The category of right $\oper{O}$-modules has a symmetric monoidal structure such that the restriction functor
\[\on{Fun}(\cat{O}\op,\S)\to\on{Fun}(\Sigma\op,\S)\]
is symmetric monoidal when the target is equipped with the Day tensor product.
\end{prop}

\begin{proof}
We have the following identity for three symmetric sequences in $\S$ (see \cite{fressemodules} 2.2.3.):
\[(M\otimes N)\circ P\cong (M\otimes P)\circ (N\otimes P)\]
If $P$ is an operad, this identity gives a right $P$-module structure on the tensor product $M\otimes N$. 
\end{proof}

The category $\Mod_{\oper{O}}$ is a symmetric monoidal category tensored over $\S$. Therefore if $\oper{P}$ is another operad, we can talk about the category $\Mod_\oper{O}[\oper{P}]$.

It is easy to check that the category $\Mod_\oper{O}[\oper{P}]$ is isomorphic to the category of $\oper{P}$-$\oper{O}$-bimodules in the category of symmetric sequences in $\S$.

From now on, we assume that $\cat{C}$ is cocomplete and that the tensor product preserves colimits in both variables.

Any right module $R$ over a single-object operad $\oper{O}$ gives rise to a functor $\cat{C}[\oper{O}]\to\cat{C}$
\[A\mapsto R\circ_{\oper{O}}A=\on{coeq}(R\circ\oper{O}(A)\rightrightarrows R(A))\]

\begin{prop}
There is an isomorphism
\[R\circ_{\oper{O}}A\cong R\otimes_{\cat{O}}A\]
\hfill$\square$
\end{prop}

\begin{prop}\label{coend}
Let $\alpha:\oper{M}\to\oper{N}$ a map of operads, the forgetful functor $\cat{C}[\oper{N}]\to\cat{C}[\oper{M}]$ has a left adjoint $\alpha_!$.

For $A\in\cat{C}[\oper{M}]$, the value at the object $n$ of $\on{Col}(\oper{N})$ of $\alpha_!A$ is given by
\[\alpha_!A(n)=\cat{N}(\alpha(-),n)\otimes_{\cat{M}}A(-)\]
\hfill$\square$
\end{prop}

\begin{proof}
See \cite[Proposition 1.15]{horelfactorization}.
\end{proof}

\begin{prop}\label{P structure}
Let $R$ be a $\oper{P}$-algebra in $\Mod_\oper{O}$. The functor $A\mapsto R\circ_{\oper{O}}A$ factors through the forgetful functor $\cat{C}[\oper{P}]\to\cat{C}$.\hfill$\square$
\end{prop}

\subsection{Homotopy theory of operads and modules}

\begin{defi}
A map $f:\oper{M}\to\oper{N}$ is said to be an equivalence, if 
\begin{itemize}
\item For any finite collection of objects $\{x_i\}$ in $\on{Ob}(\oper{M})$ and $y$ in $\on{Ob}(\oper{M})$, the induced map
\[\oper{M}(\{x_i\};y)\to\oper{N}(\{f(x_i)\};f(y))\]
is a weak equivalence.
\item The induced map $\on{Ho}(\oper{M}^{(1)})\to\on{Ho}(\oper{N}^{(1)})$ is essentially surjective.
\end{itemize}
\end{defi}

We write $\cat{C}$ for $\Mod_E$ the category of right modules over a commutative monoid in symmetric spectra. We write $\cat{C}$ in order to emphasize that the results hold more generally. However the argument are slightly different in each cases. For instance, one could work in $\cat{Ch_*}(R)$, the category of chain complexes over a commutative $\mathbb{Q}$-algebra $R$. However, this category is not stricly speaking a simplicial category. The functor $X\mapsto C_*(X,R)$ however is lax monoidal and in many respects the category $\cat{Ch_*}(R)$ behaves as a simplicial category for the ``simplicial'' structure given by 
\[\Map(C_*,D_*)_n=\cat{Ch_*}(R)(C_*\otimes C_*(\Delta[n]),D_*)\]

Similarly, our result remain true for symmetric monoidal model categories like $\S$, $s\Mod_R$, the category of simplicial modules over a commutative ring $R$ or $s\cat{Sh}(\cat{T})$ the category of simplicial sheaves over a site $\cat{T}$ with its injective model structure. However, in those cases, one has to restrict to $\Sigma$-cofibrant operads and right modules. 

If $E$ is a commutative monoid in the category $\Spec$ of symmetric spectra, we define $\Mod_E$ to be the category of right modules over $E$ equipped with the positive model structure (see \cite{schwedeuntitled} for a definition of the positive model structure). This category is a closed symmetric monoidal left proper simplicial model category. There is another model structure called the absolute model structure $\Mod_E^a$ on the same category with the same weak equivalences but more cofibrations. In particular, the unit $E$ is cofibrant in $\Mod_E^a$ but not in $\Mod_E$. The model category $\Mod_E^a$ is also a symmetric monoidal left proper simplicial model category. 

The following two theorems can be found in \cite{pavlovsymmetric}.

\begin{theo}\label{theo-operads in ring spectra}
Let $E$ be a commutative symmetric ring spectrum. Then the positive model structure on $\Mod_E$ is such that for any operad $\oper{M}$, the category $\Mod_E[\oper{M}]$ has a model structure in which the weak equivalences and fibrations are colorwise. Moreover if $A$ is a cofibrant algebra over an operad $\oper{M}$, then $A$ is cofibrant for the absolute model structure.
\end{theo}

Moreover, this model structure is homotopy invariant:

\begin{theo}\label{theo-Quillen equivalence}
Let $\alpha:\oper{M}\to\oper{N}$ be a weak equivalence of operads. Then the adjunction
\[\alpha_!:\Mod_E[\oper{M}]\leftrightarrows\Mod_E[\oper{N}]:\alpha^*\]
is a Quillen equivalence.
\end{theo}

We want to study the homotopy invariance of coends of the form $P\otimes_{\cat{M}}A$ for $A$ an $\oper{M}$-algebra and $P$ a right module over $\oper{M}$.

\begin{prop}\label{invariance of operadic coend}
Let $\oper{M}$ be an operad and let $\cat{M}$ be the PROP associated to $\oper{M}$. Let $A:\cat{M}\to\cat{C}$ be an algebra. Then
\begin{enumerate}
\item Let $P:\cat{M}\op\to \S$ be a right module. Then $P\otimes_{\cat{M}}-$ preserves weak equivalences between cofibrant $\oper{M}$-algebras.
\item If $A$ is a cofibrant algebra, the functor $-\otimes_{\cat{M}}A$ is a left Quillen functor from right modules over $\oper{M}$ to $\cat{C}$ with the absolute model structure.
\item Moreover the functor $-\otimes_{\cat{M}}A$ preserves all weak equivalences between right modules.
\end{enumerate}
\end{prop}

\begin{proof}
See \cite[Proposition 2.8.]{horelfactorization}
\end{proof}

Given a map of operad $\alpha:\oper{M}\to\oper{N}$, proposition \ref{coend} insures that the operadic left Kan extension $\alpha_!$ applied to an algebra $A$ over $\oper{M}$ coincides with the left Kan extension of the functor $A:\cat{M}\to\cat{C}$. We call the latter the categorical left Kan extension of $A$. 

The following propostion insures that the left derived functors of these two functors coincide.

\begin{prop}\label{Operadic vs categorical}
Let $\alpha:\oper{M}\to\oper{N}$ be a morphism of operads. Let $A$ be an algebra over $\oper{M}$. The derived operadic left Kan extension $\L\alpha_!(A)$ is weakly equivalent to the homotopy left Kan extension of $A:\cat{M}\to\cat{C}$ along the induced map $\cat{M}\to\cat{N}$.
\end{prop}

\begin{proof}
See \cite[Proposition 2.9.]{horelfactorization}.
\end{proof}

\bibliographystyle{alpha}
\bibliography{biblio}

\end{document}